\documentclass{article}

\usepackage[dvipdfmx]{graphicx}
\usepackage[T1]{fontenc}
\usepackage{mathtools}
\usepackage{amsmath}
\usepackage[margin=20truemm]{geometry}
\usepackage{empheq} 
\usepackage{color}
\usepackage{amssymb}
\usepackage{fancybox}
\usepackage{here}
\usepackage{bm}
\usepackage{enumitem}
\usepackage{subcaption}
\usepackage{algorithmic}
\usepackage{algorithm}
\usepackage{amsthm}
\usepackage{ulem}
\usepackage{hyperref}
\usepackage{authblk}
\usepackage[numbers]{natbib}

\newtheorem{theorem}{Theorem}
\newtheorem{lemma}{Lemma}
\newtheorem{corollary}{Corollary}
\newtheorem{definition}{Definition}
\newtheorem{assumption}{Assumption}
\newtheorem{proposition}{Proposition}
\newtheorem{remark}{Remark}

\newcommand*{\E}{\mathbb{E}}
\newcommand*{\lelabeled}[1]{\overset{\text{#1}}{\le}}
\newcommand*{\eqlabeled}[1]{\overset{\text{#1}}{=}}
\newcommand*{\gelabeled}[1]{\overset{\text{#1}}{\ge}}
\newcommand*{\leeqref}[1]{\overset{\eqref{#1}}{\le}}
\newcommand*{\geeqref}[1]{\overset{\eqref{#1}}{\ge}}
\newcommand*{\eqeqref}[1]{\overset{\eqref{#1}}{=}}
\newcommand*{\trace}[1]{\mathrm{tr}\left({#1}\right)}

\providecommand{\keywords}[1]{\textbf{Keywords:} #1}

\begin{document}

\title{Zeroth-order Random Subspace Algorithm for Non-smooth Convex Optimization}
\author[1]{Ryota Nozawa \footnote{nozawa-ryota860@g.ecc.u-tokyo.ac.jp}}
\author[2]{Pierre-Louis Poirion \footnote{pierre-louis.poirion@riken.jp}}
\author[1,2]{Akiko Takeda \footnote{takeda@mist.i.u-tokyo.ac.jp}}

\affil[1]{Department of Mathematical Informatics, The University of Tokyo, Tokyo, Japan}
\affil[2]{Center for Advanced Intelligence Project, RIKEN, Tokyo, Japan}

\maketitle

\begin{abstract}
   Zeroth-order optimization, which does not use derivative information, is one of the significant research areas in the field of mathematical optimization and machine learning. 
   Although various studies have explored zeroth-order algorithms, one of the theoretical limitations is that oracle complexity depends on the dimension,
   i.e., on the number of variables, of the optimization problem.
   In this paper, to reduce the dependency of the dimension in oracle complexity, we propose a zeroth-order random subspace algorithm by combining a gradient-free algorithm  (specifically, Gaussian randomized smoothing with central differences)
   with random projection.
   We derive the worst-case oracle complexity of our proposed method in non-smooth and convex settings; {\color{black} it is
   equivalent to standard results for full-dimensional non-smooth convex algorithms. 
   Furthermore,} we prove that ours also has a local convergence rate independent of the original dimension under additional assumptions.
   In addition to the theoretical results, numerical experiments show that when an objective function has a specific structure, the proposed method can become experimentally more efficient due to random projection.
\end{abstract}


\section{Introduction}
We consider the following unconstrained optimization problem:
\begin{equation}\label{main problem}
	\min_{x \in \mathbb{R}^n} f(x),
\end{equation}
where the objective function, $f$, is possibly non-smooth, but it is convex. Throughout this paper, we assume that $f$ is $L$-Lipschitz continuous, i.e.,
\[
  |f(x) - f(y)| \leq L \|x-y\|_2
\] 
holds for all $x$ and $y$.
Furthermore, we assume that \eqref{main problem} is zeroth-order optimization problem, which means that only the function values $f(x)$ are accessible and the derivatives of $f$ are unavailable, or that the calculation of $\nabla f$ is expensive.
Zeroth-order optimization has many applications such as bandit~\cite{bubeck2017kernel}, adversarial attack~\cite{ilyas2018black}, or hyperparameter tuning~\cite{bergstra2011algorithms}.
Additionally, there are various types of optimization methods~\cite{larson2019derivative} for zeroth-order optimization such as random search and model based methods, and these methods have been widely studied in both a smooth setting~\cite{ye2018hessian,berglund2022zeroth,kozak2023zeroth,yue2023zeroth} and a non-smooth setting~\cite{gasnikov2016gradient,rando2023optimal,duchi2015optimal,gasnikov2022power,beznosikov2020derivative,shamir2017optimal}.

However, one of the theoretical limitations of zeroth-order algorithms is that the oracle complexity depends on the dimension $n$.
For example, Duchi et al.~\cite{duchi2015optimal} and Gasnikov et al.~\cite{gasnikov2016gradient} proposed variants of mirror descent algorithms. 
Duchi et al.~\cite{duchi2015optimal} 
prove under the zeroth-order setting that a lower bound on the oracle complexity required for their method to find an $\varepsilon$-approximate solution is $O(\frac{n}{\varepsilon^2})$ for the case of linear losses.
Gasnikov et al.~\cite{gasnikov2016gradient} analyze their method on the unit simplex $S_n = \{x|x_i \ge 0, \sum_{i=1}^n x_i =1\}$ and prove that the oracle complexity is $O(\frac{n}{\varepsilon^2})$. 
On the other hand, Nesterov and Spokoiny~\cite{nesterov2017random} proposed a method using Gaussian smoothing, which is defined by
\[
  f_{\mu}(x) := \E_{u} \left[f(x+\mu u)\right].
\]
By approximating the gradient of the smoothed function $f_{\mu}$ by {\color{black}  the forward difference: $\frac{f(x+\mu u) - f(x)}{\mu} u$ or the central difference: $\frac{f(x+\mu u) - f(x-\mu u)}{2\mu} u$}, they obtain the oracle complexity under different settings; concretely, they obtain a complexity of order  $O\left(\frac{n^2}{\varepsilon^2}\right)$ when the objective function is non-smooth and convex.
Later, some papers using random smoothing~\cite{shamir2017optimal,gasnikov2022power,rando2023optimal} are able to improve oracle complexity with $O(\frac{n}{\varepsilon^2})$ using the central difference.
An oracle complexity with $O(\frac{n}{\varepsilon^2})$ is a natural result in zeroth-order optimization, since standard subgradient methods for non-smooth convex functions require $O(\frac{1}{\varepsilon^2})$ iterations and the approximation of the gradients using the finite difference needs $O(n)$ times function evaluations.

To overcome this dependence on the dimension $n$, several research~\cite{yue2023zeroth,golovin2019gradientless,qian2016derivative} assume that $f$ has a low-dimensional structure. 
For example, Yue et al.~\cite{yue2023zeroth} use a notion of effective dimensionality and prove that the oracle complexity depends on the effective dimension $\mathrm{ED}_{\alpha}:= \sup_{x}\sum_{i=1}^n \sigma_i(\nabla^2 f(x))^{\alpha}$, where $\sigma_i$ denotes a singular value, rather than on $n$.
When the objective function is convex with $L_1$-Lipschitz gradient and $H$-Lipschitz Hessian,
their algorithm is shown to have an oracle complexity where  $n$ in the oracle complexity of \cite{nesterov2017random} has changed to $\mathrm{ED}_{1/2}$.
In practice, the effective dimension is often significantly smaller than the dimension $n$.
In such a case, the oracle complexity is improved under convex and Lipschitz gradient settings.

An alternative approach to reduce the dependency on the dimension in the complexity is to use random projections~\cite{cartis2023scalable,kozak2023zeroth,berglund2022zeroth,roberts2023direct}.
Cartis and Roberts~\cite{cartis2023scalable} combine random projections with a model based derivative-free method, which approximates the objective function by interpolation.
In their approach, they solve $f(x_k+ P_ku)$ using a smaller-sized variable $u \in \mathbb{R}^d$ in each iteration, constructed with a random matrix $P_k \in \mathbb{R}^{n \times d}$ and $x_k$, instead of the original function $f(x)$.
Using random projection theory, they prove that when the objective is smooth and non-convex, the methods reduce the dimensionality of oracle complexity from $O(\frac{n^3}{\varepsilon^2})$ to $O(\frac{n^2}{\varepsilon^2})$, in order to find an $\varepsilon$-stationary point.
Kozak et al.~\cite{kozak2023zeroth} consider a zeroth-order variant of \cite{kozak2021stochastic}, approximating the exact gradient by the random projected gradient. 
They obtain the iteration complexity under various parameter choices and assumptions in the smooth setting. 
In the subsequent work, Rando et al.~\cite{rando2023optimal} propose a variant of \cite{kozak2023zeroth} and obtain an oracle complexity of order $O(\frac{n}{\varepsilon^2})$ in the non-smooth setting.
Berglund et al.~\cite{berglund2022zeroth} propose a zeroth-order variant of the randomized subspace Newton method~\cite{gower2019rsn} and prove iteration complexity in the strongly convex case.
Roberts and Royer~\cite{roberts2023direct} propose a subspace variant of the direct search methods~\cite{gratton2015direct,kolda2003optimization}.
They obtain some convergence guarantees under a wide range of subspace projections and directions, and show that their randomized methods give better results than the original full-dimensional ones.
However, to the best of our knowledge, there is no research which reduces the dependency on the dimension $n$ in oracle complexity for non-smooth functions. 

{\color{black} \subsection{Main Contribution}}
In this paper, we aim to reduce the dependency on the dimension in the worst-case oracle complexity by employing random projections, specifically under a non-smooth and convex setting.
We propose an algorithm which combines {\color{black}Gaussian smoothing using central differences}~\cite{nesterov2017random} 
and random projections. 
Our idea is to apply the Gaussian smoothing to the objective function restricted to a subspace, i.e., ${\color{black} h^{(k)}}(u):=f(x_k + P_ku)$, instead of the original function $f(x)$.
We prove that our algorithm achieves an oracle complexity of $O(\frac{n}{\varepsilon^2})$ globally, which is the standard result under the non-smooth and convex setting.
Moreover, under additional local  assumptions on  $f$, we prove an oracle complexity of $O(\frac{d^2}{\varepsilon^2})$ locally, where $d$ is the dimension of the random subspace, defined by $P_k$.
This indicates that by choosing $d$ much smaller than $\sqrt{n}$, our proposed method improves the local oracle complexity.

{\color{black}
We can summarize our contribution as follows.
\begin{itemize}
\item We propose a zeroth-order random subspace algorithm by using random projection technique
  to a Gaussian smoothing algorithm for non-smooth convex optimization problems.
\item 
  Our algorithm achieves an oracle complexity of $O(\frac{n}{\varepsilon^2})$ globally, and also has a local convergence rate independent of the original dimension under additional assumptions.
\item Our numerical experiments show that 
  the proposed method performs well due to random projection for an objective function with a specific structure.
\end{itemize}
}

{
\color{black} \subsection{Related Works on $L_2$ randomized smoothing}
Recently, $L_2$ randomized smoothing has been actively studied for smoothing non-smooth function $f$~\cite{duchi2015optimal,gasnikov2022power,shamir2017optimal}.
The random variable $u$ that defines $f_\mu$ is assumed to be a random vector uniformly distributed on a ball with center $0$ and radius $\mu$, instead of a normally distributed random Gaussian vector.
For the $L_2$ randomized smoothing, Shamir~\cite{shamir2017optimal} proposed the algorithm for bandit convex optimization and showed the optimal rate for convex Lipschitz functions using central differences.
Gasnikov et al.~\cite{gasnikov2022power} proposed the generic approach that combines smoothing and first-order methods. They show that the approach achieves the optimal rate of zeroth-order methods and can utilize various techniques in first-order methods for non-smooth zeroth-order optimization.

In fact, our proposed method can be modified so as to use a $L_2$ randomized smoothing instead of a Gaussian one, and theoretical guarantees are essentially the same for both methods (see Remark~\ref{remark:l2 randomized}).
In any case, since we use properties of Gaussian random matrices to reduce the dimension of the problem,
 we use the Gaussian smoothing in this paper for the simplicity of our discussion.
}

{\color{black} \subsection{Organization}}

In Section~\ref{seciton:Preliminaries}, we introduce some properties of the smoothing function and random matrices and vectors for our analysis.
In Section~\ref{seciton:algorithm}, we present our algorithm, and in Section~\ref{section:global convergence} we prove global convergence in $O(\frac{n}{\varepsilon^2})$. In Section~\ref{section:local convergence}, we prove local convergence in $O(\frac{d^2}{\varepsilon^2})$.
In Section~\ref{seciton:experiments},
we show numerical results and demonstrate that when the objective has a structure suitable for random projections, our method converges faster than existing methods, by reducing the function evaluation time.

\section{Preliminaries}
\label{seciton:Preliminaries}
\subsection{Notation}
  $x^*$ denotes one of the optimal solutions of \eqref{main problem}. 
  Let  $\E_{X}$ denote the expectation of a random variable $X$,
  and $\mathcal{N}(u,\Sigma)$ denote the normal distribution with mean $u$ and covariance $\Sigma$.
  $I_d$ denotes the identity matrix of size $d$.

  We use $\|\cdot\|$ as the Euclidean norm and $\|\cdot\|_{\psi_2}$ as the sub-Gaussian norm of a sub-Gaussian random variable, which is defined by
  \[
    \|X\|_{\psi_2} = \inf\{s > 0 | \E_{X}[\exp{(X^2/s^2)}] \le 2\}.
  \] 
  From the property of the sub-Gaussian norm, $\|X\|_{\psi_2} \le C$ is equivalent to
  \begin{equation}
    \label{eq:property of subgaussian}
    \mathrm{Prob}(|X| \ge t) \le 2\exp{(-ct^2/C)},
  \end{equation}
  where $c$ is an absolute constant.
  Let $\lambda_i(A)$ denote the $i$-th largest eigenvalue of a matrix $A$ and let $\mathbf{1}_{\mathcal{X}}(x)$ denote the indicator function defined by
  \[
    \mathbf{1}_{\mathcal{X}}(x) = \left\{
      \begin{array}{cc}
        1 & (x\in \mathcal{X}),\\
        0 & (\mathrm{otherwise}).
      \end{array}
    \right. 
  \]
  In particular, when $\mathcal{X} = \{x|\langle x,u\rangle \ge 0\}$ or $\mathcal{X} = \{x|\langle x,u\rangle < 0\}$ for some $u$, we use $\mathbf{1}_u^+(x)$ or $\mathbf{1}_u^-(x)$, respectively.
  $\partial f(x)$ denotes sub-differential at $x$.

\subsection{Gaussian Smoothing Function}
In this subsection, we introduce the definition of Gaussian smoothing function and recall its properties.
\begin{definition}~(e.g., \cite{nesterov2017random})
  {\label{def:smoothing function}}
  The Gaussian smoothing of $f:\mathbb{R}^n \rightarrow \mathbb{R}$ is defined by
  \begin{equation} \label{eq:smoothing}   
    f_{\mu}(x):= \E_{u}[f(x + \mu u)],
  \end{equation}
  where $u\sim \mathcal{N}(0,I_n)$ and $\mu$ is some positive constant. 
\end{definition}
It is well-known that if $f$ is convex, then $f_{\mu}$ is also convex.
As derived from Definition~\ref{def:smoothing function}, the gradient $\nabla f_{\mu}$ can be calculated by the following:
\begin{eqnarray}
  \label{eq:gradient of smoothing function}
  \E_{u} \left[\frac{f(x + \mu u)}{\mu} u \right] 
  = && \E_{u} \left[\frac{f(x + \mu u) - f(x)}{\mu} u \right]  \notag \\
  = && \E_{u} \left[\frac{f(x + \mu u) - f(x - \mu u)}{2\mu} u \right] 
  = \nabla f_{\mu}(x).
\end{eqnarray}
We can evaluate the error bound between $f$ and $f_\mu$ when $f$ is convex and Lipschitz continuous.
\begin{lemma}{\label{lemma:relations f and f_mu}}~\cite{nesterov2017random}
  If a function $f:\mathbb{R}^n \rightarrow \mathbb{R}$ is convex and $L$-Lipschitz continuous,
  then
  \begin{equation}
    \label{eq:error smoothing function}  
    f(x) \le f_{\mu}(x) \le f(x) + \mu L \sqrt{n}
  \end{equation}
  holds for any positive $\mu$.
\end{lemma}

\subsection{Random Matrices and Vectors}
In this subsection, we introduce some properties of Gaussian random matrices and vectors. 
\begin{lemma}
  {\label{lemma:properties of random vectors}}
  Let $u \in \mathbb{R}^d$ be sampled from $\mathcal{N}(0,I_d)$.
  \begin{enumerate}
    \item~\cite[Lemma~1]{nesterov2017random}
      \label{lemma:norm gaussian}
      For $p \in [0,2]$, we have
      \begin{equation}
        \label{eq:bound of expectation of norm1}
        \E_{u} \left[\|u\|^p \right]\le d^{p/2}.
      \end{equation}
      If $p \ge 2$, then we have the two-side bounds
      \begin{equation}
        \label{eq:bound of expectation of norm2}
        d^{p/2} \le \E_{u} \left[\|u\|^p\right] \le (d + p)^{p/2}. 
      \end{equation}
    \item~\cite{magnus1978moments}
      \label{lemma:trace expectation}
      Let $A \in \mathbb{R}^{d \times d}$ be a symmetric matrix. Then,
      \begin{eqnarray}
        \label{eq:trace expectation1}
        &\E_{u}[u^\top A u] = \trace{A},\\
        \label{eq:trace expectation2}
        &\E_u[(u^\top A u)^2] = (\trace{A})^2 + 2 \trace{A^2}
      \end{eqnarray}
      hold.
    \item~\cite{vershynin2018high}
      \label{lemma:gaussian concentration}
      Consider an $L$-Lipschitz function $f:\mathbb{R}^n\to\mathbb{R}$. Then,
      \[
        \|f(u) - \E_{u}\left[f(u)\right]\|_{\psi_2} \le C L
      \]
      holds, where $C$ is an absolute constant.
  \end{enumerate}
\end{lemma}
In particular, when $p=2$, the following relationship is derived from simple calculations: 
\begin{equation}
  \label{eq:squared norm expectation}
  \E_{u}\left[\|u\|^2\right] = d.
\end{equation} 
From Lemma~\ref{lemma:properties of random vectors}.\ref{lemma:gaussian concentration}, we can evaluate the sub-Gaussian norm $\|f(u) - \E_{u}[f(u)]\|_{\psi_2}$ for random variables $u\sim \mathcal{N}(0,\mu^2I_n)$ as follows.
\begin{corollary}{\label{corollary:gaussian concentration}}
  Consider a random vector $u\sim \mathcal{N}(0,I_n)$ and an $L$-Lipschitz function $f:\mathbb{R}^n\to\mathbb{R}$.
  Then,
  \[
    \|f(\mu u) - \E_{u}\left[f(\mu u)\right]\|_{\psi_2} \le C \mu L
  \]
  holds.
\end{corollary}
  
From Corollary~\ref{corollary:gaussian concentration} and the property of the sub-Gaussian norm~\eqref{eq:property of subgaussian}, we have
\begin{equation}
  \label{eq:gaussian concentration prob}
  \mathrm{Prob}(|f(\mu u) - \E_{u}\left[f(\mu u)\right]| \ge t) \le 2 \exp{\left(-\frac{ct^2}{\mu^2L^2}\right)},
\end{equation}
where the constant $c$ is independent of $n$.
Next, we recall some properties of random matrices.
\begin{lemma}
  \label{lemma:properties of random matrix}
  Let $P\in \mathbb{R}^{n\times d}(d < n)$ be a random matrix whose entries are independently drawn from $\mathcal{N}(0,1)$.
  \begin{enumerate}
    \item{\label{lemma:Johnson}}~\cite{vershynin2018high}
      Then for any $x\in \mathbb{R}^n$ and $\varepsilon\in (0,1)$, we have 
      \[
        \mathrm{Prob}\left[(1-\varepsilon)\|x\|^2\le \frac{1}{d}\|P^\top x\|^2\le (1+\varepsilon)\|x\|^2\right]\ge 1-2\exp{(-C_0\varepsilon^2 d)},
      \]
      where $C_0$ is an absolute constant.
    \item {\label{lemma: PPTx}}
      Then for any $x\in \mathbb{R}^n$, we have
      \begin{equation}
        {\label{eq: PPT x}}
        \E_{P}\left[\|PP^\top x\|^2 \right]\le 2(n+4)(d+4)\|x\|^2.
      \end{equation}
    \item ~\cite[Theorem~I\hspace{-1.2pt}I.13]{davidson2001local}
      \label{lemma:minimum singular value}
      Let $\beta = d/n$ with $d \le n$. Then for any $t > 0$,
      \[
        \mathrm{Prob}\left(\sigma_{\min}\left(\frac{1}{\sqrt{n}}P\right) \le 1 - \sqrt{\beta} -t\right) < \exp(-nt^2/2)
      \]
      holds, where $\sigma_{\min}$ denotes the minimum nonzero singular value of a matrix.
  \end{enumerate}
\end{lemma}
\begin{proof}[Proof of Lemma~\ref{lemma:properties of random matrix}.\ref{lemma: PPTx}]
  We define $P = \left(u_1,u_2,\ldots,u_n\right)^\top$, where $u_i$ is a $d$-dimensional vector. Then, we have $(PP^\top)_{ij} = u^\top_i u_j$, and therefore,
  \begin{eqnarray*}
    &&\E_{P}\left[\|PP^\top x\|^2 \right]\\
    &&=\E_{u_1,..,u_n}\left[ \sum^n_{i=1}\sum^n_{k = 1}\sum^n_{l = 1}u_l^\top u_i u_i^\top u_k x_kx_l \right]\\
    && = \E_{u_1,..,u_n}\left[\sum^n_{i=1}\sum^n_{k = 1}u_k^\top u_i u_i^\top u_k x_k^2 + \sum^n_{i=1}\sum^n_{k = 1}\sum^n_{l \neq k}u_l^\top u_i u_i^\top u_k x_kx_l\right]\\
    && = \E_{u_1,..,u_n}\left[\sum^n_{i=1}\|u_i\|^4 x_i^2 + \sum^n_{i=1}\sum^n_{k \neq i}u_k^\top u_i u_i^\top u_k x_k^2
      + \sum^n_{i=1}\sum^n_{l \neq i}u_l^\top u_i u_i^\top u_i x_ix_l + \sum^n_{i=1}\sum^n_{k \neq i}\sum^n_{l \neq k}u_l^\top u_i u_i^\top u_k x_kx_l\right].
  \end{eqnarray*}
  Regarding the third and fourth terms on the right-hand side, since the index $l$ is not equal to the other indices $i,k$, and given that $\E_{u_l}[u_l] = 0$, we have $\E_{u_1,..,u_n}[u_l^\top u_i u_i^\top u_i x_ix_l] = 0$ and $\E_{u_1,..,u_n}[u_l^\top u_i u_i^\top u_k x_kx_l] = 0$.
  Similarly, regarding the second term on the right-hand side, since the index $i$ is not equal to $k$ and $\E_{u_i}[u_iu_i^\top] = I$, we have  $\E_{u_i}[u_k^\top u_i u_i^\top u_k x_k^2] = \|u_k\|^2 x_k^2$. 
  Hence, we obtain
  \begin{eqnarray*}
    &\E_{P}\left[\|PP^\top x\|^2\right] 
    & = \E_{u_1,..,u_n}\left[\sum^n_{i=1}\|u_i\|^4 x_i^2 + \sum^n_{i=1}\sum^n_{k \neq i}\|u_k\|^2 x_k^2\right]\\
    && \eqlabeled{\eqref{eq:squared norm expectation}} \E_{u_1,..,u_n}\left[\sum^n_{i=1}\|u_i\|^4 x_i^2 + \sum^n_{i=1}\sum^n_{k \neq i}d x_k^2\right]\\
    && \lelabeled{\eqref{eq:bound of expectation of norm2}} \sum^n_{i=1}(d+4)^2 x_i^2 + \sum^n_{i=1}\sum^n_{k \neq i} dx_k^2\\
    && = ((d+4)^2 + d(n-1) )\|x\|^2\\
    && {\color{black} \lelabeled{($d \le n,\; 1 \le n$)} ((d+4)(n+4) + (d+4)(n+4) )\|x\|^2}\\
    && {\color{black} =} 2(n+4)(d+4)\|x\|^2.
  \end{eqnarray*}
  \qed
\end{proof}

\color{black}
\begin{remark}
  \label{remark:l2 randomized}
  $L_2$ randomized smoothing is defined by \eqref{eq:smoothing} with 
  the random vector $u$ sampled from the uniform distribution on the sphere of radius $1$ (i.e., $\|u\|_2=1$), and it is
  known in \cite{duchi2012randomized,gasnikov2022power} to have a dimension-independent upper bound for $f_{\mu}(x)$ as
\begin{equation}
    \label{eq:L2smoothing_function} 
f(x) \leq f_{\mu}(x) \leq f(x) + \mu L.
 \end{equation}
The difference between \eqref{eq:error smoothing function} and \eqref{eq:L2smoothing_function}
comes from the norm of random vectors. In terms of upper bounds of $f_{\mu}(x)$, there is no essential difference between these two smoothing methods.

Indeed, our oracle complexity analysis also holds when the random vector $u$ is sampled from the uniform distribution on the sphere of radius $\sqrt{n}$. 
  We can prove Lemmas~\ref{lemma:relations f and f_mu} and \ref{lemma:properties of random vectors} for a vector $u$ uniformly sampled on the sphere of radius $\sqrt{n}$. 
  For example,  
  \[
    f(x) \le \E_{u}f(x + \mu u) \le f(x) + \mu L \sqrt{n}
  \]
  holds from \cite{duchi2012randomized,gasnikov2022power}.
  Instead of Lemma~\ref{lemma:properties of random vectors}.\ref{lemma:norm gaussian} we can directly state $\|u\| = \sqrt{n}$ and equation~\eqref{eq:trace expectation1} also holds.
  For equation~\eqref{eq:trace expectation2}, $\E_{u}[(u^\top A u)^2] = (\trace{A})^2$ holds.
  Lemma~\ref{lemma:properties of random vectors}.\ref{lemma:gaussian concentration} holds for the uniform distribution on the sphere from \cite[Theorem~5.1.4]{vershynin2018high}.
  Then, the arguments in the following sections hold when we use these lemmas instead of Gaussian versions that are described in this section.
\end{remark}
\color{black}

\section{Proposed Algorithm}
\label{seciton:algorithm}
In this section, we describe the Randomized Gradient-Free method (RGF)~\cite{nesterov2017random} and our proposed method.
  RGF, a random search method, updates $x_k$ by \eqref{eq:update RGF}, where $g_{\mu}(x,u)$ denotes the approximation of the gradient at $x$ along a random direction $u$.
  The method can use forward differences or central differences for $g_{\mu}(x,u)$, i.e.,
  \[
    g_{\mu}(x,u) = \frac{f(x + \mu u) - f(x)}{\mu} u, \; \mbox{ or } \; g_{\mu}(x,u) = \frac{f(x + \mu u) - f(x-\mu u)}{2\mu} u,
  \]
  respectively.
  For a convex and Lipschitz continuous objective function,
  RGF using forward differences achieves iteration complexity of $O(\frac{n^2}{\varepsilon^2})$~\cite{nesterov2017random} and RGF using central differences achieves one of $O(\frac{n}{\varepsilon^2})$~\cite{shamir2017optimal}.
  {\color{black} For our proposed method, we confirm that using central differences attains better oracle complexity than using forward differences.}
  \begin{algorithm}[tb]
    \caption{Randomized gradient-free method (RGF) \cite{nesterov2017random}}
    \label{algorithm:RGF}
    \begin{algorithmic}
      \REQUIRE $x_0$, $\{\alpha_k\}$, $\{\mu_k\}$.
      \FOR {$k = 0,1,2,...$}
      \STATE Sample $u_k$ from $\mathcal{N}(0,I_n)$
      \STATE \begin{equation}\label{eq:update RGF}x_{k+1} = x_{k} - \alpha_k g_{\mu_k}(x_k,u_k)\end{equation}
      \ENDFOR 
    \end{algorithmic}
  \end{algorithm}

  Combining RGF and random projections, we propose Algorithm~\ref{algorithm:subspace gradient free}.
  \begin{algorithm}[tb]
    \caption{Subspace randomized gradient-free method}
    \label{algorithm:subspace gradient free}
    \begin{algorithmic}
      \REQUIRE $x_0$, $\{\alpha_k\}$, $\{\mu_k\}$, $d$
      \FOR {$k = 0,1,2,...$}
      \STATE Get a random matrix $P_k\in \mathbb{R}^{n\times d}$ whose entries are sampled from $\mathcal{N}(0,1)$. 
      \STATE ${\color{black} h^{(k)}}(u) := f(x_k + \frac{1}{\sqrt{n}}P_ku)$
      \STATE Sample $u_k$ from $\mathcal{N}(0,I_d)$
      \STATE \begin{equation}{\label{eq:update in algorithm}}x_{k+1} = x_{k} - \alpha_k \frac{{\color{black} h^{(k)}}(\mu_k u_k) - {\color{black} h^{(k)}}(-\mu_k u_k)}{2\mu_k}P_k u_k\end{equation}
      \ENDFOR 
    \end{algorithmic}
  \end{algorithm}
  We define ${\color{black} h^{(k)}}(u)$ as the restriction of $f$ to the random subspace,
  \begin{equation}
    \label{eq:subspace function}
    {\color{black} h^{(k)}}(u):= f(x_k + \frac{1}{\sqrt{n}}P_k u),
  \end{equation}
  where $P_k \in \mathbb{R}^{n \times d}$ is a random matrix, whose elements are sampled from $\mathcal{N}(0,1)$.
  Lemma~\ref{lemma:subspace smoothing} shows that the expectation of random matrices generates smoothing functions. 
  \begin{lemma}
    \label{lemma:subspace smoothing}
    \begin{eqnarray}
      \label{eq:subspace smoothing function}
      &\E_{P_k} [{\color{black} h^{(k)}}(\mu_k u_k)] = \E_{P_k}[{\color{black} h^{(k)}}(-\mu_k u_k)] = f_{ \frac{\mu_k \|u_k\|}{\sqrt{n}} }(x_k),\\
      \label{eq:subspace gradient of smoothing function}
      &\E_{P_k} \left[\frac{{\color{black} h^{(k)}}(\mu_k u_k)}{\mu_k}P_k u_k\right] = \E_{P_k} \left[-\frac{{\color{black} h^{(k)}}(-\mu_k u_k)}{\mu_k}P_k u_k\right] = \frac{\|u_k\|^2}{\sqrt{n}} \nabla f_{\frac{\mu_k \|u_k\|}{\sqrt{n}}} (x_k)
    \end{eqnarray}
    hold.
  \end{lemma}
  \begin{proof}
    Using rotational invariance of normal distribution, we have that $\E_{P_k} [F(P_k)] = \E_{P_k}[F(-P_k)]$ holds for any function $F$. Using this property, we obtain
    \[
      \E_{P_k} \left[{\color{black} h^{(k)}}(\mu_k u_k)\right] = \E_{P_k} \left[f(x_k + \frac{\mu_k}{\sqrt{n}}P_k u_k) \right]= \E_{P_k} \left[f(x_k - \frac{\mu_k}{\sqrt{n}}P_k u_k)\right] = \E_{P_k} \left[{\color{black} h^{(k)}}(-\mu_k u_k)\right].
    \]
    Notice that the distribution of $P_ku_k$ is given by $\mathcal{N}(0,\|u_k\|^2I_n)$.
    We can therefore replace $P_ku_k$ by $\|u_k\|z_k$, where $z_k$ is sampled from $\mathcal{N}(0,I_n)$.
    Then, we obtain
    \begin{equation}
      \label{eq:smoothing subspace function}
      \E_{P_k} [{\color{black} h^{(k)}}(\mu_k u_k)] 
      = \E_{z_k}\left[f(x_k + \frac{\mu_k\|u_k\|}{\sqrt{n}} z_k )\right]
      = f_{\frac{\mu_k\|u_k\|}{\sqrt{n}}}(x_k).
    \end{equation}
    Similarly, we have
    \begin{eqnarray*}
      \E_{P_k} \left[{\color{black} h^{(k)}}(\mu_k u_k)P_ku_k\right] & = & \E_{P_k} \left[f(x_k + \frac{\mu_k}{\sqrt{n}}P_k u_k)P_ku_k\right] \\
& = & -\E_{P_k} \left[f(x_k - \frac{\mu_k}{\sqrt{n}}P_k u_k)P_ku_k \right]= -\E_{P_k} \left[{\color{black} h^{(k)}}(-\mu_k u_k)P_ku_k\right],
    \end{eqnarray*}
    and 
    \begin{eqnarray*}
      \E_{P_k} \left[\frac{{\color{black} h^{(k)}}(\mu_k u_k)}{\mu_k}P_k u_k\right]
      & = & \E_{z_k}\left[\frac{f(x_k + \frac{\mu_k\|u_k\|}{\sqrt{n}} z_k )}{\mu_k} \|u_k\|z_k\right] \\
      & =  & \frac{\|u_k\|^2}{\sqrt{n}}\E_{z_k}\left[\frac{ f(x_k + \frac{\mu_k\|u_k\|}{\sqrt{n}} z_k ) }{ \frac{\mu_k \|u_k\|}{\sqrt{n} } } z_k\right]
      \eqlabeled{\eqref{eq:gradient of smoothing function}} \frac{\|u_k\|^2}{\sqrt{n}} \nabla f_{\frac{\mu_k\|u_k\|}{\sqrt{n}}}(x_k).
    \end{eqnarray*}
    \qed
  \end{proof}

\section{Global Convergence}
\label{section:global convergence}
  
  In this section, we prove global convergence of our proposed method for convex and Lipschitz continuous functions. We define $\mathcal{U}_k := (u_0,P_0,...,u_{k-1},P_{k-1})$.
  {\color{black}\begin{assumption}{\label{assumption:Lipschitz continuous and convex}}
    $f$ is $L$-Lipschitz continuous, i.e.,
    \[
      |f(x) - f(y)| \le L\|x-y\|
    \] holds for all $x,y$, and convex.
  \end{assumption}
  }

  \begin{theorem}
    \label{theorem:global convergence}
    Suppose that Assumption~\ref{assumption:Lipschitz continuous and convex} holds.
    Let the sequence $\{x_k\}$ be generated by Algorithm~\ref{algorithm:subspace gradient free}. Then for any $N\ge 1$, 
    \[
      \sum_{k=0}^{N-1} \alpha_k(\E_{\mathcal{U}_k}[f(x_k)] - f(x^*)) \le \frac{\sqrt{n}}{2d}\|x_0-x^*\|^2 + \frac{L(d+3)^{3/2}}{d}\sum_{k=0}^{N-1}\mu_k\alpha_k + \frac{L^2(d+4)^2(n+4)}{cd\sqrt{n}}\sum_{k=0}^{N-1}\alpha_k^2
    \]
    holds.
  \end{theorem}
  \begin{proof}
    We define $r_k := \|x_k - x^*\|$.
    From \eqref{eq:update in algorithm}, we have
    \begin{eqnarray} 
      \label{eq:eq1 in global convergence}
      r^2_{k+1}
      = r^2_k - 2\alpha_k \langle \frac{{\color{black} h^{(k)}}(\mu_k u_k) - {\color{black} h^{(k)}}(-\mu_k u_k)}{2\mu_k}P_k u_k ,x_k - x^* \rangle \notag \\
      + \alpha_k^2 \left( \frac{{\color{black} h^{(k)}}(\mu_k u_k) - {\color{black} h^{(k)}}(-\mu_ku_k)}{2\mu_k} \right)^2\|P_ku_k\|^2.
    \end{eqnarray}
    Taking the expectation with respect to $u_k$ and $P_k$, we then evaluate the second and third terms.
    Regarding the second term on the right-hand side of \eqref{eq:eq1 in global convergence}, we have
    \begin{align*}
      \E_{u_k}\E_{P_k} \left[\left\langle \frac{{\color{black} h^{(k)}}(\mu_k u_k) - {\color{black} h^{(k)}}(-\mu_k u_k)}{2\mu_k}P_k u_k ,x_k - x^* \right\rangle\right] &
      \eqeqref{eq:subspace gradient of smoothing function}  \E_{u_k}\left[\left\langle \frac{\|u_k\|^2}{\sqrt{n}} \nabla f_{\frac{\mu_k \|u_k\|}{\sqrt{n}}} (x_k) ,x_k - x^* \right\rangle\right] \\
      &\gelabeled{convexity} \E_{u_k} \left[\frac{\|u_k\|^2}{\sqrt{n}} (f_{\frac{\mu_k \|u_k\|}{\sqrt{n}}} (x_k) - f_{\frac{\mu_k \|u_k\|}{\sqrt{n}}} (x^*))\right]\\
      & \geeqref{eq:error smoothing function} \E_{u_k} \left[\frac{\|u_k\|^2}{\sqrt{n}} (f(x_k) - f(x^*) - L \mu_k \|u_k\|)\right]\\
      & \eqlabeled{\eqref{eq:squared norm expectation}} \frac{d}{\sqrt{n}} (f(x_k) - f(x^*)) - \frac{L \mu_k \E_{u_k}[\|u_k\|^3]}{n^{1/2}} \\
      & \gelabeled{\eqref{eq:bound of expectation of norm2}} \frac{d}{\sqrt{n}} (f(x_k) - f(x^*)) - \frac{L \mu_k (d+3)^{3/2}}{n^{1/2}}.
    \end{align*}
    For the third term on the right-hand side of \eqref{eq:eq1 in global convergence}, from $P_ku_k = \|u_k\|z_k\;(z_k \sim \mathcal{N}(0,I_n))$, we have
    \begin{eqnarray*}
      &\gamma& :=\E_{P_k}\E_{u_k} \left[ \left( \frac{{\color{black} h^{(k)}}(\mu_k u_k) - {\color{black} h^{(k)}}(-\mu_ku_k)}{2\mu_k} \right)^2\|P_ku_k\|^2\right]\\
      && \eqlabeled{\eqref{eq:smoothing subspace function}} \frac{1}{4\mu_k^2}\E_{z_k}\E_{u_k} \left[\left( f(x_k + \frac{\mu_k\|u_k\|}{\sqrt{n}}z_k) - f(x_k - \frac{\mu_k\|u_k\|}{\sqrt{n}}z_k) \right)^2\|u_k\|^2\|z_k\|^2\right]\\
      && = \frac{1}{4\mu_k^2}\E_{z_k}\E_{u_k} \left[\left( (f(x_k + \frac{\mu_k\|u_k\|}{\sqrt{n}}z_k) - \beta) + (\beta - f(x_k - \frac{\mu_k\|u_k\|}{\sqrt{n}}z_k)) \right)^2\|u_k\|^2\|z_k\|^2\right]. \\
    \end{eqnarray*}
    By applying the inequality $(x+y)^2\le 2x^2+2y^2$, we obtain
    \[
      \gamma \le \frac{1}{2\mu_k^2}\E_{z_k}\E_{u_k} \left[\left(\left(f(x_k + \frac{\mu_k\|u_k\|}{\sqrt{n}}z_k) - \beta\right)^2 + \left(\beta - f(x_k - \frac{\mu_k\|u_k\|}{\sqrt{n}}z_k) \right)^2\right)\|u_k\|^2\|z_k\|^2\right].
    \]
    From the rotational invariance of $z_k$, we have
    \[
      \gamma \le \frac{1}{2\mu_k^2}\E_{z_k}\E_{u_k}\left[ \left(\left(f(x_k + \frac{\mu_k\|u_k\|}{\sqrt{n}}z_k) - \beta\right)^2 + \left(\beta - f(x_k + \frac{\mu_k\|u_k\|}{\sqrt{n}}z_k) \right)^2\right)\|u_k\|^2\|z_k\|^2\right].
    \] 
    Selecting $\beta = \E_{z_k}\left[f(x_k + \frac{\mu_k\|u_k\|}{\sqrt{n}}z_k)\right]$, we have
    \begin{eqnarray*}
      &\gamma
      &\le \frac{1}{\mu_k^2}\E_{u_k}\left[\|u_k\|^2\E_{z_k} \left[\left(f(x_k + \frac{\mu_k\|u_k\|}{\sqrt{n}}z_k) - \E_{z_k}\left[f(x_k + \frac{\mu_k\|u_k\|}{\sqrt{n}}z_k)\right]\right)^2\|z_k\|^2\right]\right]\\
      &&\lelabeled{H\"older's ineq.} \frac{1}{\mu_k^2}\E_{u_k}\left[\|u_k\|^2\sqrt{\E_{z_k} \left[\left(f(x_k + \frac{\mu_k\|u_k\|}{\sqrt{n}}z_k) - \E_{z_k}\left[f(x_k + \frac{\mu_k\|u_k\|}{\sqrt{n}}z_k)\right]\right)^4\right]}\sqrt{\E_{z_k}[\|z_k\|^4]}\right].
    \end{eqnarray*}
    We evaluate $\E_{z_k} \left[\left(f(x_k + \frac{\mu_k\|u_k\|}{\sqrt{n}}z_k) - \E_{z_k}\left[f(x_k + \frac{\mu_k\|u_k\|}{\sqrt{n}}z_k)\right]\right)^4\right]$ using Corollary~\ref{corollary:gaussian concentration}.
    Note that for any non-negative random variables $X$, $\E_{X}[X] = \int_0^\infty \mathrm{Prob}(X\ge t) \mathrm{d}t$ holds. Using this relation, we have
    \begin{eqnarray*}
      &&\E_{z_k} \left[\left(f(x_k + \frac{\mu_k\|u_k\|}{\sqrt{n}}z_k) - \E_{z_k}\left[f(x_k + \frac{\mu_k\|u_k\|}{\sqrt{n}}z_k)\right]\right)^4\right]\\
      && \hspace{100pt}= \int_{0}^{\infty} \mathrm{Prob}\left(\left(f(x_k + \frac{\mu_k\|u_k\|}{\sqrt{n}}z_k) - \E_{z_k}\left[f(x_k + \frac{\mu_k\|u_k\|}{\sqrt{n}}z_k)\right]\right)^4 \ge t\right) \mathrm{d}t\\
      && \hspace{100pt}= \int_{0}^{\infty} \mathrm{Prob}\left(\left|f(x_k + \frac{\mu_k\|u_k\|}{\sqrt{n}}z_k) - \E_{z_k}\left[f(x_k + \frac{\mu_k\|u_k\|}{\sqrt{n}}z_k)\right]\right| \ge t^{1/4}\right) \mathrm{d}t\\
      && \hspace{100pt} \lelabeled{\eqref{eq:gaussian concentration prob}} \int_{0}^{\infty} 2\exp{\left(-\frac{cnt^{1/2}}{\mu_k^2\|u_k\|^2L^2}\right)}\mathrm{d}t\\
      && \hspace{100pt} = \frac{\mu_k^4\|u_k\|^4L^4}{c^2n^2}\int_{0}^{\infty} 2\exp{\left(-T^{1/2}\right)}\mathrm{d}T\\
      && \hspace{100pt} = \frac{4\mu_k^4\|u_k\|^4L^4}{c^2n^2},
    \end{eqnarray*}
    where the last equality follows from $\int_0^\infty \exp{\left(-T^{1/2}\right)}\mathrm{d}T = 2$.
    Then, we obtain
    \begin{eqnarray}
      \notag
      &\gamma
      &\le \frac{1}{\mu_k^2}\E_{u_k}\left[\|u_k\|^2 \sqrt{\frac{4\mu_k^4\|u_k\|^4L^4}{c^2n^2}}\sqrt{\E_{z_k}[\|z_k\|^4]}\right]\\
      \notag
      && = \frac{2L^2}{cn} \E_{u_k}[\|u_k\|^4] \sqrt{\E_{z_k}[\|z_k\|^4]}\\
      \label{eq:evaluation of gamma}
      && \lelabeled{\eqref{eq:bound of expectation of norm2}} \frac{2L^2(d+4)^2(n+4)}{cn}.
    \end{eqnarray}
    Therefore, we obtain
    \[
      \E_{u_k}\E_{P_k} [r_{k+1}^2]
      \le r_k^2 
      -  \frac{2d\alpha_k}{\sqrt{n}} (f(x_k) - f(x^*)) 
      + \frac{2L\mu_k \alpha_k (d+3)^{3/2}}{n^{1/2}} 
      + \frac{2L^2\alpha_k^2(d+4)^2(n+4)}{cn}.
    \]
    Taking the expectation with respect to $\mathcal{U}_k$, we obtain
    \[
      \E_{\mathcal{U}_{k+1}}[r_{k+1}^2] 
      \le \E_{\mathcal{U}_{k}}[r_{k}^2]  
      -  \frac{2d\alpha_k}{\sqrt{n}} (\E_{\mathcal{U}_k}\left[f(x_k)\right] - f(x^*)) 
      + \frac{2L \mu_k \alpha_k (d+3)^{3/2}}{n^{1/2}} 
      + \frac{2L^2\alpha_k^2(d+4)^2(n+4)}{cn}.
    \]
    Summing up these inequalities from $k=0$ and $k=N-1$, we obtain
    \[
      \E_{\mathcal{U}_{N}}[r_{N}^2]
      \le r_{0}^2  
      -  \frac{2d}{\sqrt{n}} \sum_{k = 0}^{N-1} \alpha_k(\E_{\mathcal{U}_k}\left[f(x_k)\right] - f(x^*)) 
      + \frac{2L(d+3)^{3/2}}{n^{1/2}}\sum_{k=0}^{N-1} \mu_k \alpha_k  
      + \frac{2L^2(d+4)^2(n+4)}{cn} \sum_{k=0}^{N-1} \alpha_k^2. 
    \]
    Therefore, 
    \[
      \sum_{k = 0}^{N-1} \alpha_k(\E_{\mathcal{U}_k}\left[f(x_k)\right] - f(x^*)) 
      \le \frac{\sqrt{n}}{2d}r_0^2 
      +  \frac{L(d+3)^{3/2}}{d}\sum_{k=0}^{N-1} \mu_k \alpha_k  
      + \frac{L^2(d+4)^2(n+4)}{c\sqrt{n}d}  \sum_{k=0}^{N-1} \alpha_k^2. 
    \]
    \qed
  \end{proof}
  With fixed $\alpha_k = \alpha$ and $\mu_k = \mu$, we obtain
  \[
    \min_{0\le i \le N-1} \E_{\mathcal{U}_i}[f(x_i)] - f(x^*) \le \frac{\sqrt{n}}{2dN\alpha}r_0^2 +  \frac{L(d+3)^{3/2}}{d}\mu  + \frac{L^2(d+4)^2(n+4)}{cd\sqrt{n}}  \alpha.
  \]
  From this relation, the oracle complexity for achieving 
  the inequality: $\min_{0\le i \le N-1} \E_{\mathcal{U}_i}[f(x_i)] - f(x^*) \leq \varepsilon$ is 
  \[
    N = \frac{8r_0^2L^2 (n+4)(d+4)^2}{c^2d^2\varepsilon^2} = O\left(\frac{n}{\varepsilon^2}\right)
  \] 
  with
  \[
    \alpha = \frac{\sqrt{cn}r_0}{L(d+4)\sqrt{2(n+4)N}}, \;\mu \le \frac{\varepsilon d}{2L(d+3)^{3/2}}.
  \]
  These parameters are obtained by the relations of $\frac{\sqrt{n}}{2dN\alpha}r_0^2 = \frac{L^2(d+4)^2(n+4)}{cd\sqrt{n}}  \alpha =\frac{\varepsilon}{4}$ and $\frac{L(d+3)^{3/2}}{d}\mu \le \frac{\varepsilon}{2}$.

Note that in each iteration, our algorithm calculates the function value twice.
    Therefore, the oracle complexity is equal to twice the iteration complexity. 

\section{Local Convergence}
\label{section:local convergence}
In this section, we prove, under some local assumptions on $f$, local convergence of our proposed method.

\subsection{Assumptions}
\begin{assumption}
  {\label{assumption:reduced dimension assumption1}}
  We have that $d = o(n)$, and $d\to \infty$ as $n\to \infty$. 
\end{assumption}

Next we consider the following local assumptions on $f$. 
\begin{assumption}
  \label{assumption:local properties}
  There exists a neighborhood $B^*$ of $x^*$ and an Alexandrov matrix $\tilde{H}(x)$  that satisfy the following properties.
  \begin{enumerate}[label = (\roman{enumi}),ref = (\roman{enumi})]
    \item \label{assumption:relative strongly convex}
    There exist constants $\tilde{L}$ and $\tau$, and a subgradient $g\in \partial f(y)$ such that for all $x, y\in B^*$:
    \begin{eqnarray}
      \label{eq:relative convex}
      f(x) \ge f(y) + \langle g,x-y\rangle + \frac{\tau}{2}(x-y)^\top \tilde{H}(y) (x-y),\\
      \label{eq:relative smooth}
      f(x) \le f(y) + \langle g,x-y\rangle + \frac{\tilde{L}}{2}(x-y)^\top \tilde{H}(y) (x-y).
    \end{eqnarray}
    \item {\label{assumption:local rank}}
    There exist constants $\sigma \in (0,1)$ and $\bar{\lambda}>0$ such that $\lambda_{\sigma n}(\tilde{H}(x)) \ge \bar{\lambda}$ holds for all $x \in B^*$.
  \end{enumerate}
\end{assumption}

  When the objective function is twice differentiable, we set $g = \nabla f(y)$ and $\tilde{H}(y)=\nabla^2 f(y)$.
  In the smooth setting, \eqref{eq:relative convex} and \eqref{eq:relative smooth} in Assumption~\ref{assumption:local properties}\ref{assumption:relative strongly convex} are called relative convexity and smoothness~\cite{gower2019rsn,karimireddy2018global}, respectively, and some functions achieve this property (e.g. logistic function, Wasserstein distance).
  This assumption implies that the objective function $f$ is $\tau$-strongly convex and $\tilde{L}$-smooth under the semi-norm $\|\cdot\|_{\tilde{H}(x)}$ for all $x,y \in B^*$. 
  
  While non-smooth objective functions do not necessarily have gradients and Hessians at some points,
    we can show that when $f$ is convex, $f$ is twice differentiable almost everywhere.
  \begin{theorem}~\cite{alexandrov1939almost,niculescu2006convex}
    \label{theorem:convex twice differentiable}
    Every convex function $f:\mathbb{R}^n\to \mathbb{R}$ is twice differentiable almost everywhere in the following sense: $f$ is twice differentiable at $z$ with Alexandrov Hessian $\tilde{H}(z) = \tilde{H}^\top(z) \succeq 0$, if $\nabla f(z)$ exists, and if for every $\varepsilon > 0$ there exists $\delta > 0$ such that
    $\|x-z\|\le \delta$ implies
    \[
      \sup_{y\in \partial f(x)} \|y- \nabla f(z) - \tilde{H}(z)(x-z)\| \le \varepsilon\|x-z\|.
    \]
  \end{theorem}
  Assumption~\ref{assumption:local properties}\ref{assumption:relative strongly convex} is inspired by Theorem~\ref{theorem:convex twice differentiable}, because there exists $\tilde{H}(x)$ almost everywhere such that
  \[
    \lim_{h \to 0} \frac{f(x+h) - f(x) -\langle\nabla f(x),h \rangle -\frac12 h^\top \tilde{H}(x)h }{\|h\|^2} = 0.
  \]
  Note that the subgradient $g$ and the matrix $\tilde{H}(x)$ are used only in the analysis, not in our algorithm.

In the following subsection, we assume the above two assumptions in addition to
Assumption~\ref{assumption:Lipschitz continuous and convex}.
  The theoretical results in this section can only hold locally around an optimal solution $x^*$.
  Indeed, we assume Lipschitz continuity as Assumption~\ref{assumption:Lipschitz continuous and convex}
  and relative convexity as Assumption~\ref{assumption:local properties}\ref{assumption:relative strongly convex}.
  This implies that $f(x) - f(x^*) = O(\|x - x^*\|)$ and $f(x) - f(x^*) = \Omega(\|x - x^*\|^2)$ hold, and then as $\|x - x^*\| \to \infty$, these assumptions conflict. 

\subsection{Local Theoretical Guarantees}
We define ${\color{black} h^{(k)}_{\mu_k} }(u)$ as the smoothing function on the random subspace:
\begin{equation}
  \label{eq:smoothing subspace function for u}
  {\color{black} h^{(k)}_{\mu_k} }(u) := \E_{u_k} [{\color{black} h^{(k)}}(u + \mu_k u_k)] = \E_{u_k} \left[f\left(x_k + \frac{1}{\sqrt{n}}P_k(u + \mu_k u_k)\right)\right].
\end{equation}

From \eqref{eq:gradient of smoothing function}, we have
\begin{eqnarray}
  \label{eq:gradient of smoothing subspace function for u}
  &\nabla {\color{black} h^{(k)}_{\mu_k} }(u) = \E_{u_k} \left[\frac{{\color{black} h^{(k)}}(u + \mu_k u_k) - {\color{black} h^{(k)}}(u)}{\mu_k} u_k\right] = \E_{u_k} \left[\frac{{\color{black} h^{(k)}}(u + \mu_k u_k) - {\color{black} h^{(k)}}(u -\mu_k u_k)}{2\mu_k} u_k\right]. \hspace{20pt}
\end{eqnarray}

Under some assumptions, we can show that $P ^\top \tilde{H} (x) P$ is a positive definite matrix with high probability.
\begin{proposition}
  \cite[Proposition~5.4]{fuji2022randomized}
  {\label{prop:strongly convexify}}
  Let $0<\varepsilon_0 < 1$. Then under Assumptions~\ref{assumption:Lipschitz continuous and convex}, \ref{assumption:reduced dimension assumption1} and \ref{assumption:local properties}\ref{assumption:local rank}, there exists $n_0 \in \mathbb{N}$
  (which depends only on $\varepsilon_0$ and $\sigma$) such that if $n \ge n_0$, for any $x \in B^*$,
  \[
    P ^\top \tilde{H} (x) P  \succeq \frac{(1-\varepsilon_0)^2 n}{2}\sigma^2 \bar{\lambda} I_d
  \]
  holds with probability at least $1-6\exp{(-d)}$.
\end{proposition}

Now, we prove local convergence of our proposed method.
\begin{theorem}
  \label{theorem:local convergence}
  Let $0 < \varepsilon_0 < 1$ and $n_0$ be as defined in Proposition~\ref{prop:strongly convexify}, and the sequence $\{x_k\}$ be generated by Algorithm~\ref{algorithm:subspace gradient free}.
  Suppose that Assumptions~\ref{assumption:Lipschitz continuous and convex}, \ref{assumption:reduced dimension assumption1}, and \ref{assumption:local properties} hold.
  If $n \ge n_0$, $x_k \in B^*$, and $\mu_k \le \frac{1}{\trace{P_k^\top \tilde{H}(x_k) P_k}}$ hold for any $k \ge 0$, then for any $N\ge 1$,
  at least one of the following holds:
  \begin{enumerate}
    \item \begin{eqnarray}{\label{eq:local result1}}\sum_{k = 0}^{N-1} \alpha_k(\E_{\mathcal{U}_k}[f(x_k)] - f(x^*)) \le \frac{1}{2}r_0^2 +  L\sqrt{d}\sum_{k=0}^{N-1}  \alpha_k\mu_k + \frac{L^2(n+4)(d+4)^2}{cn}  \sum_{k=0}^{N-1} \alpha_k^2,\hspace{30pt}\end{eqnarray}
    \item \begin{equation}{\label{eq:local result2}}\min_{0\le k \le N-1}f(x_k) - f(x^*) \le LC_1 \sqrt{{\color{black}\frac{8d+41}{d^2}}} + \frac{LC_2}{n\sqrt{d}},\end{equation}
  \end{enumerate} 
  where $C_1 := \frac{8L}{(1-\varepsilon_0)^2 \tau C \sigma^2 \bar{\lambda}}$, $C_2:= \frac{2\sqrt{3}\tilde{L}}{(1-\varepsilon_0)^2 \tau C \sigma^2 \bar{\lambda}}$, and $C := 1 - 6\exp(-d) - 2\exp(-\frac{C_0d}{4})$.
\end{theorem}
\begin{proof}
  Let $g_k \in \partial f(x_k)$.
  Using the same argument as in the proof of Theorem~\ref{theorem:global convergence}, from \eqref{eq:eq1 in global convergence} and \eqref{eq:evaluation of gamma},
  we obtain
  \begin{eqnarray}
    \notag
    &\E_{u_k}\E_{P_k} [r_{k+1}^2] &
    \le r_k^2 -  2\alpha_k \E_{u_k}\E_{P_k}\left[ \left\langle\frac{{\color{black} h^{(k)}}(\mu_k u_k) - {\color{black} h^{(k)}}(-\mu_k u_k)}{2\mu_k}P_k u_k ,x_k - x^* \right\rangle\right] + \frac{2L^2\alpha_k^2}{cn} (d+4)^2(n+4)\\
    \label{eq:local first inequality}
    & & \eqlabeled{\eqref{eq:gradient of smoothing subspace function for u}} r_k^2 -  2\alpha_k \E_{P_k} \left[\langle \nabla {\color{black} h^{(k)}_{\mu_k} }(0),P_k^\top(x_k - x^*) \rangle\right] + \frac{2L^2\alpha_k^2}{cn} (d+4)^2(n+4).
  \end{eqnarray}
  We reevaluate the second term on the right-hand side.
  Now, we evaluate the error between $\E_{u_k}[\langle \nabla {\color{black} h^{(k)}_{\mu_k} }(0),u \rangle]$ and $\langle g_k,\frac{1}{\sqrt{n}}P_ku \rangle$ for any $u$.
  From relation~\eqref{eq:relative convex} with $x = x_k + \frac{1}{\sqrt{n}}P_k u + \frac{\mu_k}{\sqrt{n}}P_k u_k$ and $y = x_k $,
  we have
  \[
    f(x_k + \frac{1}{\sqrt{n}}P_ku + \frac{\mu_k}{\sqrt{n}}P_k u_k) \ge f(x_k) + \langle g_k, \frac{1}{\sqrt{n}}P_ku + \frac{\mu_k}{\sqrt{n}}P_k u_k\rangle + \frac{\tau}{2n}(u+\mu_k u_k)^\top P_k^\top \tilde{H}( x_k )P_k (u+\mu_k u_k).
  \]
  Taking the expectation with respect to $u_k$, we obtain
  \begin{eqnarray}
    \notag
    &{\color{black} h^{(k)}_{\mu_k} }(u) 
    &\ge f(x_k) + \E_{u_k}\left[\left\langle g_k, \frac{1}{\sqrt{n}}P_ku + \frac{\mu_k}{\sqrt{n}}P_k u_k\right\rangle\right] + \frac{\tau}{2n}u^\top P_k^\top \tilde{H}( x_k )P_k u + \E_{u_k}\left[\frac{\tau\mu_k}{n} u_k^\top P_k^\top \tilde{H}( x_k )P_k u\right]\\
    \notag
    && \hspace{100pt} + \E_{u_k}\left[\frac{\tau\mu_k^2}{2n} u_k^\top P_k^\top \tilde{H}( x_k )P_k u_k\right] \\
    \notag
    &&\eqlabeled{$\E_{u_k}[u_k] = 0$} f(x_k) +  \langle g_k, \frac{1}{\sqrt{n}}P_ku  \rangle +   \frac{\tau}{2n}u^\top P_k^\top \tilde{H}( x_k )P_k u + \E_{u_k}\left[\frac{\tau\mu_k^2}{2n} u_k^\top P_k^\top \tilde{H}( x_k )P_k u_k\right] \\
    \label{eq:eq1}
    && \eqlabeled{\eqref{eq:trace expectation1}} f(x_k) +  \langle g_k, \frac{1}{\sqrt{n}}P_ku  \rangle +   \frac{\tau}{2n}u^\top P_k^\top \tilde{H}( x_k )P_k u +\frac{\tau\mu_k^2}{2n} \trace{P_k^\top \tilde{H}( x_k )P_k}.
  \end{eqnarray}
  First, we evaluate $\langle g_k, \frac{1}{\sqrt{n}}P_ku  \rangle$. 
  Using $\E_{u_k}[u_ku_k^\top] = I_d$, we have
  \begin{eqnarray}
    &\langle g_k, \frac{1}{\sqrt{n}}P_ku  \rangle &
    \notag
    = \E_{u_k}\langle g_k, \frac{1}{\sqrt{n}}P_ku_ku_k^\top u  \rangle\\
    \notag
    & &= \E_{u_k} \left[\langle g_k, \frac{1}{\sqrt{n}}P_ku_k\rangle \langle u_k,u\rangle \right]\\
    \label{eq:evalueta gTPu}
    & &= \E_{u_k} \left[\langle g_k, \frac{1}{\sqrt{n}}P_ku_k\rangle \langle u_k,u\rangle \mathbf{1}_{u}^+(u_k)\right]
    +\E_{u_k} \left[\langle g_k, \frac{1}{\sqrt{n}}P_ku_k\rangle \langle u_k,u\rangle \mathbf{1}_{u}^-(u_k)\right].
    \hspace{30pt}
  \end{eqnarray}
  We compute both an upper bound and a lower bound for $\langle g_k, \frac{1}{\sqrt{n}}P_ku_k\rangle$.
  From the convexity of $f$, we have 
  \begin{equation}
    \label{eq:upper bound1}
    \frac{f(x_k + \frac{\mu_k}{\sqrt{n}}P_k u_k) - f(x_k)}{\mu_k} 
    \ge \langle g_k,\frac{1}{\sqrt{n}}P_k u_k\rangle.
  \end{equation}
  From relation~(\ref{eq:relative smooth}) with $x = x_k + \frac{\mu_k}{\sqrt{n}}P_ku_k$ and $y = x_k$, we have
  \begin{equation}
    \label{eq:lower bound1}
    \langle g_k,\frac{1}{\sqrt{n}}P_k u_k\rangle + \frac{\tilde{L}\mu_k}{2n} u_k^\top P_k^\top \tilde{H}(x_k) P_k u_k
    \ge \frac{f(x_k + \frac{\mu_k}{\sqrt{n}}P_k u_k) - f(x_k)}{\mu_k}.
  \end{equation}
  Using these relations, we have
  \begin{eqnarray}
    \notag
    &\langle g_k, \frac{1}{\sqrt{n}}P_ku  \rangle &
    \eqlabeled{\eqref{eq:evalueta gTPu}} \E_{u_k} \left[\langle g_k, \frac{1}{\sqrt{n}}P_ku_k\rangle \langle u_k,u\rangle \mathbf{1}_{u}^+(u_k)\right]
    +\E_{u_k} \left[\langle g_k, \frac{1}{\sqrt{n}}P_ku_k\rangle \langle u_k,u\rangle \mathbf{1}_{u}^-(u_k)\right]\\
    \notag
    &&\gelabeled{\eqref{eq:upper bound1},\eqref{eq:lower bound1}} \E_{u_k} \left[\frac{f(x_k + \frac{\mu_k}{\sqrt{n}}P_k u_k) - f(x_k)}{\mu_k} \langle u_k,u\rangle \mathbf{1}_{u}^+(u_k)\right]
    - \E_{u_k}\left[ \frac{\tilde{L}\mu_k}{2n} u_k^\top P_k^\top \tilde{H}(x) P_k u_k  \langle u_k,u\rangle \mathbf{1}_{u}^+(u_k)\right]\\
    \notag
    &&\hspace{100pt}+ \E_{u_k}\left[\frac{f(x_k + \frac{\mu_k}{\sqrt{n}}P_k u_k) - f(x_k)}{\mu_k} \langle u_k,u\rangle \mathbf{1}_{u}^{-}(u_k)\right]\\
    \notag
    && = \E_{u_k} \left[\frac{f(x_k + \frac{\mu_k}{\sqrt{n}}P_k u_k) - f(x_k)}{\mu_k} \langle u_k,u\rangle\right] - \E_{u_k}\left[ \frac{\tilde{L}\mu_k}{2n} u_k^\top P_k^\top \tilde{H}(x) P_k u_k  \langle u_k,u\rangle \mathbf{1}_{u}^+(u_k)\right]\\
    \notag
    && = \E_{u_k} \left[\frac{{\color{black} h^{(k)}}(\mu_ku_k) - {\color{black} h^{(k)}}(0)}{\mu_k} \langle u_k,u\rangle \right]- \E_{u_k} \left[\frac{\tilde{L}\mu_k}{2n} u_k^\top P_k^\top \tilde{H}(x) P_k u_k  \langle u_k,u\rangle \mathbf{1}_{u}^{+}(u_k)\right]\\
    \label{eq:evaluate1}
    && \eqlabeled{\eqref{eq:gradient of smoothing subspace function for u}} \langle\nabla {\color{black} h^{(k)}_{\mu_k} }(0),u\rangle - \E_{u_k} \left[\frac{\tilde{L}\mu_k}{2n} u_k^\top P_k^\top \tilde{H}(x) P_k u_k  \langle u_k,u\rangle \mathbf{1}_{u}^+(u_k)\right].
  \end{eqnarray}
  Regarding the second term, using $\E_{u_k} [F(u_k)] = \E_{u_k} [F(-u_k)]$ from the rotational invariance of the normal distribution, we have
  \begin{eqnarray}
    \notag
    &\E_{u_k}  \left[u_k^\top P_k^\top \tilde{H}(x) P_k u_k  \langle u_k,u\rangle \mathbf{1}_{u}^+(u_k)\right]&
    = \frac{1}{2}\E_{u_k}  \left[u_k^\top P_k^\top \tilde{H}(x) P_k u_k  \langle u_k,u\rangle \mathbf{1}_{u}^+(u_k)\right]\\
    \notag
    && \hspace{100pt}+ \frac{1}{2}\E_{u_k}  \left[u_k^\top P_k^\top \tilde{H}(x) P_k u_k  \langle u_k,u\rangle \mathbf{1}_{u}^+(u_k)\right]\\
    \notag
    &&\eqlabeled{rotation invariant} \frac{1}{2}\E_{u_k}  \left[u_k^\top P_k^\top \tilde{H}(x) P_k u_k  \langle u_k,u\rangle \mathbf{1}_{u}^+(u_k)\right]\\
    \notag
    && \hspace{100pt} - \frac{1}{2}\E_{u_k}  \left[u_k^\top P_k^\top \tilde{H}(x) P_k u_k  \langle u_k,u\rangle \mathbf{1}_{u}^+(-u_k)\right]\\
    \notag
    && = \frac{1}{2}\E_{u_k} \left[ u_k^\top P_k^\top \tilde{H}(x) P_k u_k  \langle u_k,u\rangle \mathbf{1}_{u}^+(u_k)\right]\\
    \notag
    && \hspace{100pt} - \frac{1}{2}\E_{u_k} \left[ u_k^\top P_k^\top \tilde{H}(x) P_k u_k  \langle u_k,u\rangle \mathbf{1}_{u}^-(u_k)\right]\\
    \label{eq:evaluate2}
    && = \frac{1}{2}\E_{u_k} \left[ u_k^\top P_k^\top \tilde{H}(x) P_k u_k  |\langle u_k,u\rangle| \right].
  \end{eqnarray}
  Now, we evaluate $\E_{u_k}\left[ u_k^\top P_k^\top \tilde{H}(x) P_k u_k | \langle u_k,u\rangle|\right]$.
  From H\"older's inequality and Lemma~\ref{lemma:properties of random vectors}.\ref{lemma:trace expectation},
  we obtain
  \begin{eqnarray*}
    &\E_{u_k} \left[u_k^\top P_k^\top \tilde{H}(x) P_k u_k | \langle u_k,u\rangle|\right]
    &\lelabeled{H\"older's ineq.} \sqrt{\E_{u_k} \left[(u_k^\top P_k^\top \tilde{H}(x) P_k u_k)^2\right]}\sqrt{\E_{u_k} \left[| \langle u_k,u\rangle|^2\right]}\\
    && = \sqrt{\E_{u_k} \left[(u_k^\top P_k^\top \tilde{H}(x) P_k u_k)^2\right]}\sqrt{\E_{u_k} \left[u^\top u_ku_k^\top u \right]}\\
    && \eqlabeled{$\E_{u_k}[u_ku_k^\top] = I_d$} \sqrt{\E_{u_k} \left[(u_k^\top P_k^\top \tilde{H}(x) P_k u_k)^2\right]}\sqrt{\|u\|^2}\\
    &&\lelabeled{\eqref{eq:trace expectation2}}  \sqrt{ (\trace{P_k^\top \tilde{H}(x) P_k})^2 + 2\trace{(P_k^\top \tilde{H}(x) P_k)^2}} \|u\|.\\
  \end{eqnarray*}
  For any positive semidefinite matrix $A$, $\trace{A^2} = \sum_{i}\lambda_i(A)^2 \le (\sum_i \lambda_i(A))^2 = (\trace{A})^2$ holds. Then, we have
  \begin{equation}
    \label{eq:evaluate3}
    \E_{u_k} \left[u_k^\top P_k^\top \tilde{H}(x) P_k u_k | \langle u_k,u\rangle|\right] \le \sqrt{3} \trace{P_k^\top \tilde{H}(x) P_k} \|u\|.
  \end{equation}
  Finally, from \eqref{eq:evaluate1},\eqref{eq:evaluate2}, and \eqref{eq:evaluate3}, we obtain
  \begin{equation}
    {\label{eq:product eval}}
      \langle g_k, \frac{1}{\sqrt{n}}P_ku  \rangle 
    \ge \langle \nabla {\color{black} h^{(k)}_{\mu_k} }(0),u \rangle -  \frac{\sqrt{3}\tilde{L}\mu_k}{4n} \trace{P_k^\top \tilde{H}(x) P_k} \|u\|.
  \end{equation}
  Combining relations from \eqref{eq:eq1} and \eqref{eq:product eval},
  we obtain
  \[
    {\color{black} h^{(k)}_{\mu_k} }(u) \ge f(x_k) + \langle \nabla {\color{black} h^{(k)}_{\mu_k} }(0),u \rangle -  \frac{\sqrt{3}\tilde{L}\mu_k}{4n} \trace{P_k^\top \tilde{H}(x) P_k} \|u\| 
    + \frac{\tau}{2n}u^\top P_k^\top \tilde{H}( x_k )P_k u +  \frac{\mu_k^2\tau}{2n}\trace{P_k^\top \tilde{H}( x_k )P_k }.
  \]
  By substituting $-P_k^\top(x_k - x^*)$ for $u$, 
  \begin{eqnarray}
    \notag
      {\color{black} h^{(k)}_{\mu_k} }(-P_k^\top(x_k - x^*)) \ge f(x_k) - \langle \nabla {\color{black} h^{(k)}_{\mu_k} }(0),P_k^\top(x_k - x^*) \rangle -  \frac{\sqrt{3}\tilde{L}\mu_k}{4n} \trace{P_k^\top \tilde{H}(x) P_k} \|P_k^\top(x_k - x^*)\| \\
    \notag
      + \frac{\tau}{2n}(x_k - x^*)^\top P_kP_k^\top \tilde{H}( x_k )P_k P_k^\top (x_k-x^*) +  \frac{\mu_k^2\tau}{2n}\trace{P_k^\top \tilde{H}( x_k )P_k }
  \end{eqnarray}
  holds, and then, regarding the second term on the right-hand side of \eqref{eq:local first inequality}, we have
  \begin{eqnarray}
    \notag
    \langle\nabla {\color{black} h^{(k)}_{\mu_k} }(0),P_k^\top(x_k - x^*) \rangle
    \ge f(x_k) - {\color{black} h^{(k)}_{\mu_k} }(-P_k^\top(x_k - x^*))  -  \frac{\sqrt{3}\tilde{L}\mu_k}{4n} \trace{P_k^\top \tilde{H}(x) P_k} \|P_k^\top(x_k - x^*)\| \\
    \label{eq:evaluate local main}
  + \frac{\tau}{2n}(x_k - x^*)^\top P_kP_k^\top \tilde{H}( x_k )P_k P_k^\top (x_k-x^*) +  \frac{\mu_k^2\tau}{2n}\trace{P_k^\top \tilde{H}( x_k )P_k }. \hspace{30pt}
  \end{eqnarray}
  After taking the expectation with respect to $P_k$, we evaluate the right-hand side. As for ${\color{black} h^{(k)}_{\mu_k} }(-P_k^\top(x_k - x^*))$, we have
  \begin{eqnarray*}
    &\E_{P_k}[{\color{black} h^{(k)}_{\mu_k} }(-P_k^\top(x_k - x^*))] &
    \eqlabeled{\eqref{eq:smoothing subspace function for u}} \E_{P_k}\E_{u_k} \left[f(x_k - \frac{1}{\sqrt{n}}P_kP_k^\top(x_k - x^*) + \frac{\mu_k}{\sqrt{n}}P_k u_k)\right]\\
    &&\lelabeled{Lipschitz} \E_{P_k}\left[ f(x_k - \frac{1}{\sqrt{n}}P_kP_k^\top(x_k - x^*)) \right]+ \E_{P_k}\E_{u_k}\left[\frac{L\mu_k}{\sqrt{n}}\|P_k u_k\|\right]\\
    &&\lelabeled{Lipschitz} f(x^*) + L\E_{P_k}\left[\|(I_n - \frac{1}{\sqrt{n}}P_kP_k^\top)(x_k - x^*)\|\right] + \E_{P_k}\E_{u_k}\left[\frac{L\mu_k}{\sqrt{n}}\|P_k u_k\|\right].
  \end{eqnarray*}
  Applying Lemma~\ref{lemma:properties of random vectors}.\ref{lemma:norm gaussian}, we have $\E_{P_k}\E_{u_k}\left[\|P_k u_k\|\right] = \E_{z_k}\E_{u_k}\left[\|z_k\|\|u_k\|\right] \le \sqrt{nd}$ and
  \begin{eqnarray}
    \notag
    &\E_{P_k}\left[\sqrt{\|(I_n - \frac{1}{\sqrt{n}}P_kP_k^\top)(x_k - x^*)\|^2}\right]
    &\lelabeled{Jensen's ineq.} \sqrt{\E_{P_k}\left[\|(I_n - \frac{1}{\sqrt{n}}P_kP_k^\top)(x_k - x^*)\|^2\right]}\\
    \notag
    &&= \sqrt{\E_{P_k} \left[ \|x_k-x^*\|^2 - \frac{2}{\sqrt{n}}\|P_k^\top(x_k - x^*)\|^2 + \frac{1}{n} \|P_kP_k^\top(x_k - x^*)\|^2\right]}\\
    \notag
    && \eqlabeled{$\E_{P_k} [P_kP_k^\top] = dI_n$} \sqrt{ \|x_k-x^*\|^2 - \frac{2d}{\sqrt{n}}\|x_k - x^*\|^2 + \E_{P_k}\left[\frac{1}{n} \|P_kP_k^\top(x_k - x^*)\|^2\right]}\\
    \notag
    &&\lelabeled{\eqref{eq: PPT x}} \sqrt{1 - 2\frac{d}{\sqrt{n}} + 2\frac{(d+4)(n+4)}{n} }~\|x_k - x^*\|\\
    \notag
    && = \sqrt{2d + 9 + \frac{8d}{n} + \frac{32}{n} - \frac{2d}{\sqrt{n}}}~\|x_k - x^*\|\\
    \notag
    && {\color{black}= \sqrt{2d + 9 + \frac{6d}{n} + \frac{32}{n} - 2d\left( \frac{1}{\sqrt{n}} -\frac1n\right)}~\|x_k - x^*\|}\\
    \notag
    && {\color{black}\lelabeled{$(d \ge 0,\; n\ge 1)$} \sqrt{8d + 41}~\|x_k - x^*\|.}
  \end{eqnarray}
  Finally, we have
  \begin{equation}
    \label{eq:evaluate4}
    \E_{P_k}[{\color{black} h^{(k)}_{\mu_k} }(-P_k^\top(x_k - x^*))]
    \le f(x^*)  + L\mu_k\sqrt{d} + L{\color{black}\sqrt{8d + 41}}~\|x_k - x^*\|.
  \end{equation}
  
  For $\frac{\sqrt{3}\tilde{L}\mu_k}{4n} \trace{P_k^\top \tilde{H}(x) P_k} \|P_k^\top(x_k - x^*)\|$ in \eqref{eq:evaluate local main}, 
  from $\mu_k \le \frac{1}{\trace{P_k^\top \tilde{H}(x_k) P_k}}$ and $\E_{P_k}[\|P_k^\top (x_k - x^*)\|] \le \sqrt{d}\|x_k - x^*\|$, we have
  \begin{equation}  
    \label{eq:evaluate5}
    \E_{P_k}\left[\frac{\sqrt{3}\tilde{L}\mu_k}{4n} \trace{P_k^\top \tilde{H}(x) P_k} \|P_k^\top(x_k - x^*)\| \right]
    \le \frac{\sqrt{3d}\tilde{L}}{4n} \|x_k - x^*\|.
  \end{equation}
  
  Next, regarding $(x_k - x^*)^\top P_kP_k^\top \tilde{H}( x_k )P_k P_k^\top (x_k-x^*)$, we have
  \begin{eqnarray*}
    \E_{P_k} [(x_k - x^*)^\top P_kP_k^\top \tilde{H}( x_k )P_k P_k^\top (x_k-x^*)]
    \ge \E_{P_k} [\lambda_{\min}(P_k^\top \tilde{H}( x_k )P_k) \|P_k^\top (x_k-x^*)\|^2].
  \end{eqnarray*}
  For applying conditional expectation property, we define $\mathcal{A} := \{P | \lambda_{\min}(P^\top \tilde{H}( x_k )P) \ge \frac{(1-\varepsilon_0)^2 n}{2}\sigma^2 \bar{\lambda}\}$
  , $\mathcal{B} := \{P | \|P^\top (x_k - x^*)\|^2 \ge \frac{d}{2}\|x_k - x^*\|^2\}$ and $X := (x_k - x^*)^\top P_kP_k^\top \tilde{H}( x_k )P_k P_k^\top (x_k-x^*) \ge 0$.
  Then, we have
  \[
    \E_{P_k} [X] = \E_{P_k} [X|\mathcal{A}\cap \mathcal{B}]P(\mathcal{A} \cap \mathcal{B}) + \E_{P_k} [X|\overline{\mathcal{A} \cap \mathcal{B}}](1 - P(\mathcal{A} \cap \mathcal{B}) ) 
    \ge \E_{P_k} [X|\mathcal{A}\cap \mathcal{B}]P(\mathcal{A}\cap \mathcal{B}).
  \]
  Applying Lemma~\ref{lemma:properties of random matrix}.\ref{lemma:Johnson} and Proposition~\ref{prop:strongly convexify}, we have
  \[
    P(\mathcal{A} \cap \mathcal{B}) 
    \ge 1 - P(\mathcal{A}) - P(\mathcal{B})
    \ge 1 - 6\exp{(-d)} - 2 \exp{(-\frac{C_0d}{4})}.
  \]
  Let $C := 1 - 6\exp{(-d)} - 2 \exp{(-\frac{C_0d}{4})}$, we have
  \begin{equation}
    \label{eq:evaluate6}
    \E_{P_k}[X] \ge  \frac{(1-\varepsilon_0)^2C dn}{4}\sigma^2 \bar{\lambda} \|x_k - x^*\|^2.
  \end{equation}
  From \eqref{eq:evaluate local main}, \eqref{eq:evaluate4}, \eqref{eq:evaluate5}, and \eqref{eq:evaluate6}, 
  we obtain
  \begin{eqnarray}
    \notag
    \E_{P_k} [\langle \nabla {\color{black} h^{(k)}_{\mu_k} }(0),P_k^\top(x_k - x^*) \rangle]
    \ge f(x_k) - f(x^*)  - L\mu_k\sqrt{d} - L{\color{black}\sqrt{8d+41}}\|x_k - x^*\| \\
    \label{eq:EPk (h,P(x-x*))}
    - \frac{\sqrt{3d}\tilde{L}}{4n} \|x_k - x^*\| + \frac{(1-\varepsilon_0)^2 \tau Cd}{8}\sigma^2 \bar{\lambda} \|x_k - x^*\|^2. \hspace{30pt}
  \end{eqnarray}
  To satisfy the condition:
  \[
    - L{\color{black}\sqrt{8d+41}}\|x_k - x^*\| 
    - \frac{\sqrt{3d}\tilde{L}}{4n} \|x_k - x^*\| + \frac{(1-\varepsilon_0)^2\tau Cd}{8}\sigma^2 \bar{\lambda} \|x_k - x^*\|^2
    \ge 0,
  \]
  we need
  \begin{equation}
    {\label{eq:distance x_k - x^*}}
      \|x_k - x^*\| \ge C_1 \sqrt{{\color{black}\frac{8d+41}{d^2}}} + \frac{C_2}{n\sqrt{d}}.
    \end{equation}
  When $x_k$ satisfies \eqref{eq:distance x_k - x^*}, we obtain
  \begin{equation}
    \label{eq:new eval in local}
    \E_{P_k} [\langle \nabla {\color{black} h^{(k)}_{\mu_k} }(0),P_k^\top(x_k - x^*) \rangle]
    \ge f(x_k) - f(x^*)  - L\mu_k\sqrt{d}.
  \end{equation}
  Then, we have
  \begin{eqnarray*}
    &\E_{u_k}\E_{P_k} [r_{k+1}^2] 
    & \leeqref{eq:local first inequality} r_k^2 -  2\alpha_k \E_{P_k}[ \langle \nabla {\color{black} h^{(k)}_{\mu_k} }(0),P_k^\top(x_k - x^*) \rangle] + \frac{2L^2\alpha_k^2}{cn} (d+4)^2(n+4)\\
    && \leeqref{eq:new eval in local} r_k^2 -  2\alpha_k(f(x_k) - f(x^*)) + 2L\alpha_k\mu_k\sqrt{d} + \frac{2L^2\alpha_k^2}{cn} (d+4)^2(n+4).
  \end{eqnarray*}
  Taking the expectation with respect to $\mathcal{U}_k$, we have
  \begin{eqnarray}
    \label{eq:decrease expectation norm}
    \E_{\mathcal{U}_{k+1}}[r_{k+1}^2] 
    \le \E_{\mathcal{U}_{k}}[r_{k}^2]  -  2\alpha_k (\E_{\mathcal{U}_k}[f(x_k)] - f(x^*)) + 2L\alpha_k\mu_k\sqrt{d} + \frac{2L^2\alpha_k^2}{cn} (d+4)^2(n+4).
    \hspace{30pt}
  \end{eqnarray}
  If $\{x_k\}_{k=1,..,N}$ satisfy \eqref{eq:distance x_k - x^*},
  \[
    \sum_{k = 0}^{N-1} \alpha_k(\E_{\mathcal{U}_k}[f(x_k)] - f(x^*)) \le \frac{1}{2}r_0^2 +  L\sqrt{d}\sum_{k=0}^{N-1}  \alpha_k\mu_k  + \frac{L^2}{cn} (d+4)^2(n+4) \sum_{k=0}^{N-1} \alpha_k^2. 
  \]
  When $x_k$ does not satisfy \eqref{eq:distance x_k - x^*} for some $k$, i.e.,
  \[
    \|x_k - x^*\| < C_1 \sqrt{{\color{black}\frac{8d+41}{d^2}}} + \frac{C_2}{n\sqrt{d}},
  \]
  from Lipschitz continuity of $f$, we have
  \[
    f(x_k) - f(x^*) \le LC_1 \sqrt{{\color{black}\frac{8d+41}{d^2}}} + \frac{LC_2}{n\sqrt{d}}.
  \]
  \qed
\end{proof}
With fixed $\alpha_k = \alpha$ and $\mu_k = \mu$, from \eqref{eq:local result1} we obtain
\[
  \min_{0\le i \le N-1} \E_{\mathcal{U}_i}[f(x_i)] - f(x^*) \le \frac{r_0^2}{2N\alpha} +  L\sqrt{d}\mu  + \frac{L^2(n+4)(d+4)^2}{cn}  \alpha.
\]
From this relation, the iteration complexity 
for achieving the inequality: $\min_{0\le i \le N-1} \E_{\mathcal{U}_i}[f(x_i)] - f(x^*) <\varepsilon$ is 
\begin{equation}
  \label{eq:local iteration complexity(central)}
  N = \frac{8r_0^2 L^2(n+4)(d+4)^2}{cn\varepsilon^2} = O\left(\frac{d^2}{\varepsilon^2}\right)
\end{equation}
with
\begin{equation}
  \label{eq:step size and mu(central)}
  \alpha = \frac{\sqrt{cn}r_0}{L(d+4)\sqrt{2(n+4)N}}, \;\mu \le \frac{\varepsilon }{2L\sqrt{d}}.
\end{equation}
From \eqref{eq:local result2}, when $LC_1 \sqrt{{\color{black}\frac{8d+41}{d^2}}} \le \frac{\varepsilon}{2}$ and $\frac{LC_2}{n\sqrt{d}} \le \frac{\varepsilon}{2}$ hold, we obtain $f(x_k) - f(x^*) \le \varepsilon$.
Then, the reduced dimension $d$ must satisfy $d = \Omega(\varepsilon^{-2})$.
When comparing the global and local iteration complexity, the local behavior becomes better when the original dimension $n$ satisfies $n  = \Theta(\varepsilon^{-p})$ with $p>4$.
In this case, while global iteration complexity becomes $O(n\varepsilon^{-2}) = O(\varepsilon^{-p -2})$, the local iteration complexity achieves $O(d^2\varepsilon^{-2}) = O(\varepsilon^{-6})$, which is less than $O(\varepsilon^{-p -2})$, with reduced dimension $d = \Theta(\varepsilon^{-2})$.
\begin{remark}
  Indeed, our algorithm uses only $P_ku_k$, and does not use $P_k$ and $u_k$ separately.
  This is the most interesting part of this research using random projections.
  We consider that the advantage of $P_ku_k$ comes from the relation
  \[
    \E_{P_k}[\min_{u\in \mathbb{R}^d} {\color{black} h^{(k)}}(u)=f(x_k + \frac{1}{\sqrt{n}}P_k u)] \le \E_{v_k}[\min_{\alpha \in \mathbb{R}} f(x_k + \alpha v_k)],
  \] 
  where $v_k \in  \mathbb{R}^n$ is a random vector whose entries come from $\mathcal{N}(0,1)$.
  This relation is clear because all entries of $P_k$ and $v_k$ follow $\mathcal{N}(0,1)$ and the left-hand side problem is identical to the right-hand side when $d = 1$.
  Noticing that the function on the left-hand side is denoted by ${\color{black} h^{(k)}}(u)$ in \eqref{eq:subspace function},
  we can regard \eqref{eq:update in algorithm} in Algorithm~\ref{algorithm:subspace gradient free}
  as one iteration of Algorithm~\ref{algorithm:RGF} 
  (i.e., RGF~\cite{nesterov2017random}) applied to the problem on the left-hand side. 
\end{remark}

\section{Experiments}
\label{seciton:experiments}
\subsection{Softmax Regression} \label{sec:softmax_exp}
In this section, we evaluate the performance of the proposed method, Algorithm~\ref{algorithm:subspace gradient free}, and compare it with 
RGF~\cite{nesterov2017random} with central differences, which is described as Algorithm~\ref{algorithm:RGF}, by optimizing a softmax loss function with $L_1$ regularization:
\[
  \min_{w_1,...,w_c,b_1,...,b_c} -\frac{1}{m}\sum_{i=1}^m \log{\frac{\exp(w_{y_i}^\top x_i + b_{y_i})}{\sum_{k=1}^c \exp(w_{k}^\top x_i + b_{k})} } + \lambda \sum_{k=1}^c( \|w_k\|_1 + \|b_k\|_1),
\]
where $(x_i,y_i)$ denotes data and $y_i \in \{1,2,...,c\}$.
For the $L_1$ regularization, we set $\lambda = 10^{-6}$ and use reduced dimensional size $d \in \{10,50,100\}$.
For both methods, we set a smoothing parameter $\mu_k = 10^{-8}$ and use a fixed step size $\alpha_k = 10^{i}\;(i\in \mathbb{N})$. 

\begin{table}[tb]
  \caption{Details of the datasets for Softmax regression~\cite{chang2011libsvm}.}
  \centering
  \begin{tabular}{c|ccc}
    Name & feature & class ($c$) & training size ($m$)\\
    \hline
    SCOTUS & 126,405 & 13 & 5,000\\
    news20 & 62,061 & 20 & 15,935\\
    \hline
  \end{tabular}
\label{table:experiment1}
\end{table}
\begin{figure}[tb]
  \begin{minipage}{0.5\hsize}
    \centering
    \includegraphics[width=2.6 in]{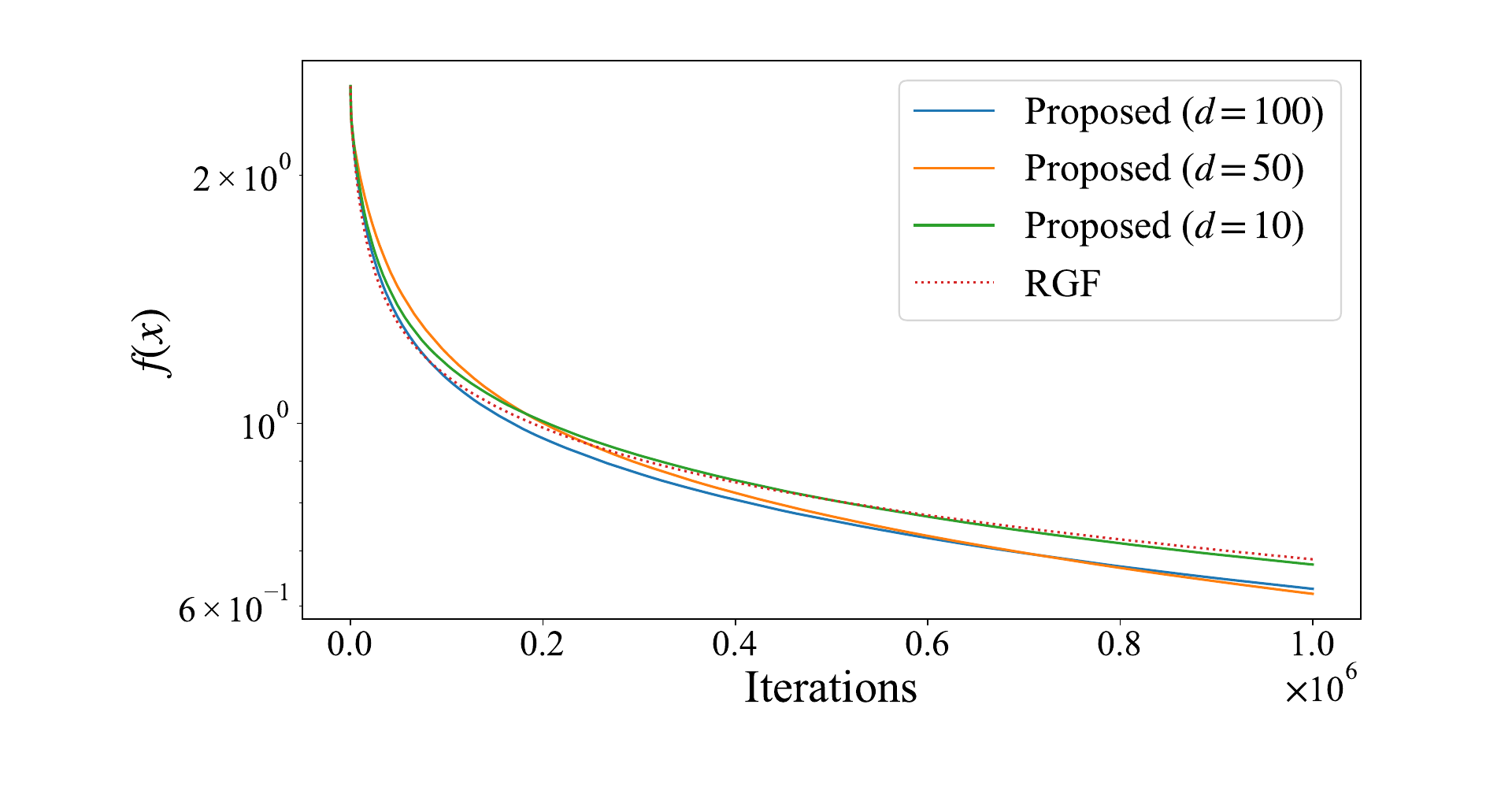}  
    \subcaption*{SCOTUS Dataset}
  \end{minipage}
  \begin{minipage}{0.5\hsize}
    \centering
    \includegraphics[width=2.6 in]{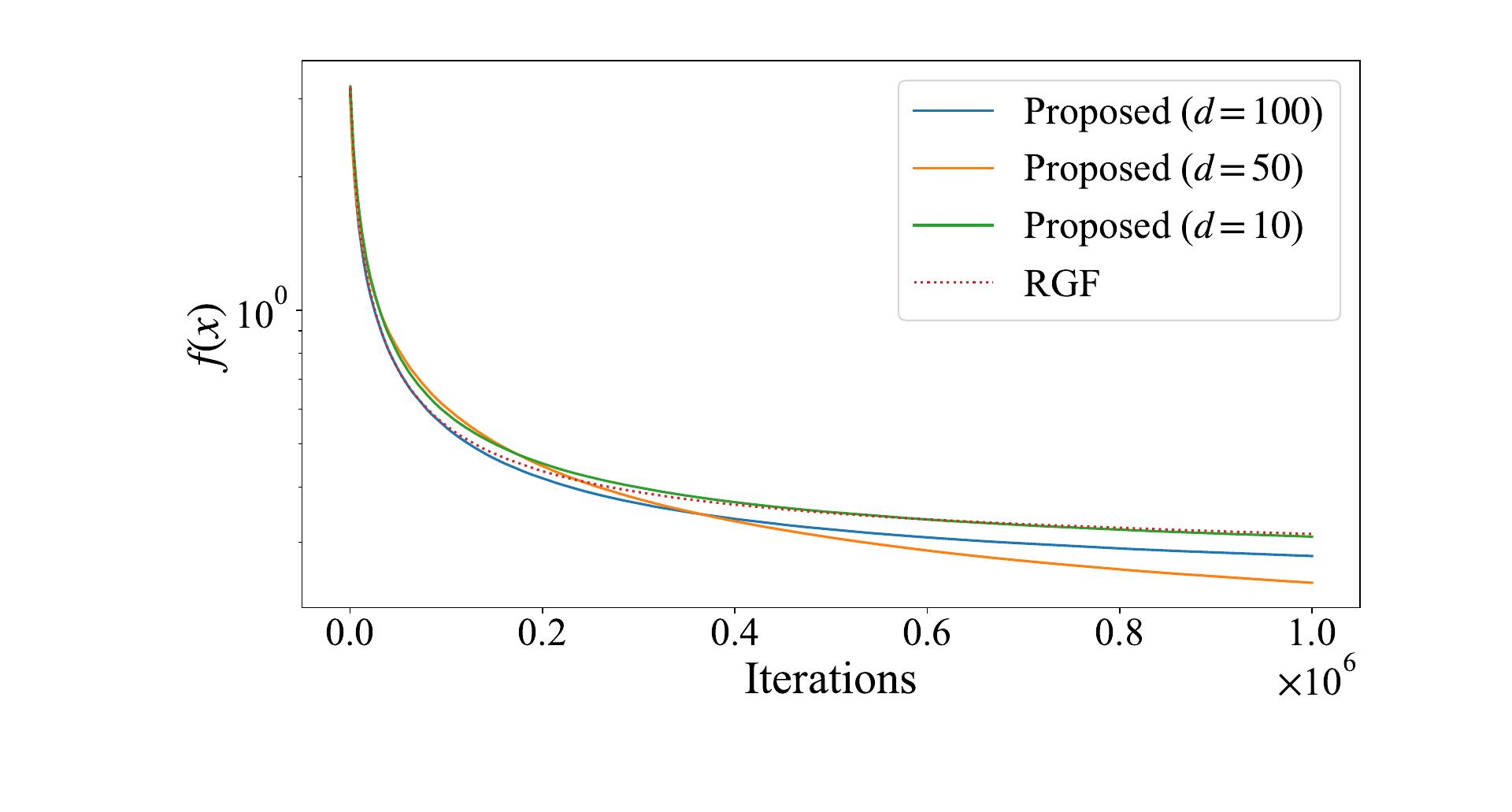}
    \subcaption*{news20 Dataset}
  \end{minipage}
  \caption{Softmax loss function with $L_1$ regularization.}
  \label{fig:experiment1}
\end{figure}

Figure~\ref{fig:experiment1} shows function values of RGF and the proposed methods per iteration for the datasets listed in Table~\ref{table:experiment1}.
From Figure~\ref{fig:experiment1}, we can find that our algorithm converges faster than RGF after a sufficient number of iterations, while RGF reduces the objective function value more rapidly in the early iterations.
{\color{black}
This behavior might be consistent with the theoretical guarantees of global and local worst-case complexities (Theorems~\ref{theorem:global convergence} and \ref{theorem:local convergence}, respectively),
considering that the global complexity $O(\frac{n}{\varepsilon^2})$ is the same to the one of RGF. 
  However,  it seems difficult to see 
  the difference between the coefficients of the local and global sublinear rates, i.e., that the local iteration complexity is $d^2/n$ times the global one.
  Perhaps the reason is that the rate improvement of local convergence is about worst-case complexity,
  and such worst-case may not always be achieved in practice.
}

  As for the result of function values in time, generating random matrices is time-consuming and there is no benefit using random projections in view of time for general functions.
  In this setting, our proposed methods spend more time for the same number of iterations.

\subsection{Adversarially Robust Logistic Regression}
We consider the next adversarially robust optimization, which is studied in~\cite{madry2017towards}:
  \begin{eqnarray}
&& \min_{w,b} \max_{\| \tilde{x}\| \le \delta} ~ \tilde{g}( \tilde{x};w,b)  := \notag \\ 
&&  -\frac{1}{m}\sum_{i=1}^m \log{\frac{1}{(1 + e^{-y_i (w^\top (x_i+\tilde{x}) + b)})}} + \lambda ( \|w\|_1 + \|b\|_1),   
 \label{problem:adversarially robust}
  \end{eqnarray}
  where $\lambda = 10^{-7}$.
By letting $\theta^\top=(w^\top,b)$ and $\eta^\top=(\tilde{x}^\top, 0)$, 
we can rewrite Problem~\eqref{problem:adversarially robust} as $\min_{\theta} f(\theta)$, where
\begin{eqnarray*}
  && f(\theta):=  \max_{\eta \in \mathcal{H}} g_\theta(\theta^\top \eta), ~~   \mathcal{H}:= \{\eta: \| \eta\| \le \delta\}, \\
  && g_\theta(\alpha) := -\frac{1}{m}\sum_{i=1}^m \log{\frac{1}{(1 + e^{-y_i (\alpha + w^\top x_i + b)})}} + \lambda ( \|w\|_1 + \|b\|_1). \\
\end{eqnarray*}
Note that the derivative of $f$ is difficult to compute due to the non-smoothness of $f$ in general.

In this formulation, we can take advantages of random projections in our proposed method.
When we evaluate the function value $f(\theta_k + \frac{1}{\sqrt{n}}P_k u)$, it is necessary to solve
\begin{equation}
  {\label{problem:max}}
  \max_{\eta\in \mathcal{H}} g_{\theta_k}(\theta_k^\top \eta + u^\top \frac{1}{\sqrt{n}}P_k^\top \eta).
\end{equation}
In this case, we solve the following approximated optimization problem:
\begin{equation}
  \label{problem:approximated problem}
  \max_{(\alpha,\beta) \in \mathcal{A}} g_{\theta_k}(\alpha + u^\top\beta),
\end{equation}
where $\mathcal{A} = \{(\alpha,\beta)| \frac{\alpha^2}{\|\theta_k\|^2} + \beta^2 \le \delta^2\}$.
We will explain that this approximated problem~\eqref{problem:approximated problem} is equivalent to the original problem~\eqref{problem:max} 
when some condition holds, and also show that the condition holds with high probability.
Now, we confirm that we can obtain $\eta$ such that $\|\eta\| \le \delta$ from the solution of the approximated problem~\eqref{problem:approximated problem}.
Let $(\alpha^*,\beta^*)$ denote optimal solutions of this approximated problem~\eqref{problem:approximated problem}. When  $d+1 < n$, we can confirm the existence of $\eta$ that satisfies $\alpha^*=\theta_k^\top \eta$ and $\beta^*=\frac{1}{\sqrt{n}}P_k^\top \eta$ by solving the linear equation;
\begin{equation}
  \label{eq:linear equation}
  z^*
  = A \eta,\; 
  A := \left(
    \begin{array}{c}
      \frac{\theta_k^\top}{\|\theta_k\|}\\
      \frac{1}{\sqrt{n}}P_k^\top
    \end{array}
  \right),\;
  z^* :=
  \left(
    \begin{array}{c}
      \frac{\alpha^*}{\|\theta_k\|}\\
      \beta^*
    \end{array}
  \right).
\end{equation}
From the linear dependence of row vectors in $A$, the minimum norm solution of \eqref{eq:linear equation} is $\eta = A^\top (AA^\top)^{-1}z^*$, and then $\|\eta\|^2 = (z^*)^\top (AA^\top)^{-1}z^* \le \lambda_{\max}((AA^\top)^{-1})\|z^*\|^2$ holds.
This inequality implies that when $\|z^*\|\le \sqrt{\lambda_{\min}(AA^\top)}\delta$ holds, $\|\eta\| \le \delta$ also holds.
From $(\alpha^*,\beta^*) \in \mathcal{A}$, $\|z^*\|^2 = \left(\frac{\alpha^*}{\|\theta_k\|}\right)^2 + \|\beta^*\|^2 \le \delta^2$ holds. 
Then, if $\lambda_{\min}(AA^\top) \ge 1$ holds, $\|z^*\| \le \delta$ implies $\|\eta\| \le \delta$.
To show this relation, we prove next Lemma~\ref{lemma:singular value of A}.
\begin{lemma}
  \label{lemma:singular value of A}
  Let $A := \left(
    \begin{array}{cc}
      \frac{x}{\|x\|},
      \frac{1}{\sqrt{n}}P
    \end{array}
  \right)^\top$, where $P\in \mathbb{R}^{n\times d}$ is a random matrix whose entries are sampled from $\mathcal{N}(0,1)$ and $x \in \mathbb{R}^n$.
  Then, $\lambda_{\min}(AA^\top) \ge  1 - 2(3+\varepsilon)\sqrt{\frac{d}{n}} + \frac{4d}{n}$ holds with probability at least $1 - 2\exp{(-C_0\varepsilon^2 d)}  - \exp{(-d/2)} $.
\end{lemma}
\begin{proof}
  Let $w_1\in \mathbb{R}$ and $w_2 \in \mathbb{R}^d$. 
  From the property of the minimum eigenvalue, we have
  \begin{eqnarray*}
    &\lambda_{\min}(AA^\top)&
    = \min_{\|w\|=1} w^\top AA^\top w
    = \min_{w_1^2 + \|w_2\|^2=1} w_1^2 + \frac{2}{\sqrt{n}\|x\|} w_1 x^\top P w_2 + \frac{1}{n}w_2^\top P^\top P w_2\\
    &&\ge\min_{w_1^2 + \|w_2\|^2=1} w_1^2 - \frac{2}{\sqrt{n}\|x\|} |w_1| \|P^\top x\| \|w_2\| + \frac{1}{n}w_2^\top P^\top P w_2.
  \end{eqnarray*}
  Regarding the second term on the right-hand side, from Lemma~\ref{lemma:properties of random matrix}.\ref{lemma:Johnson} and $w_1^2 + \|w_2\|^2 =1$, we obtain
  \[
    \frac{2}{\sqrt{n}\|x\|} |w_1| \|P^\top x\| \|w_2\|
    \lelabeled{Lemma~\ref{lemma:properties of random matrix}.\ref{lemma:Johnson}} 2(1+\varepsilon)\sqrt{\frac{d}{n}} |w_1|\|w_2\| \le 2(1+\varepsilon)\sqrt{\frac{d}{n}}
  \]
  with probability at least $1 - 2\exp{(-C_0\varepsilon^2 d)}$. Then, we have
  \[
    \lambda_{\min}(AA^\top)
    \ge \min_{w_1^2 + \|w_2\|^2=1}(w_1^2 + \frac{1}{n}w_2^\top P^\top P w_2) - 2(1+\varepsilon)\sqrt{\frac{d}{n}}.
  \]
  The first term on right-hand side is equivalent to the minimum eigenvalue of 
  \[
    \left(
      \begin{array}{cc}
        1& 0\\
        0& \frac{1}{n}P^\top P
      \end{array}
    \right),
  \]
  and then, we have $\min_{w_1^2 + \|w_2\|^2=1}(w_1^2 + \frac{1}{n}w_2^\top P^\top P w_2) = \min\{1,\sigma_d(\frac{1}{\sqrt{n}}P)^2\}$. 
  Applying Lemma~\ref{lemma:properties of random matrix}.\ref{lemma:minimum singular value}, $\sigma_d(\frac{1}{\sqrt{n}}P)\ge 1 - 2\sqrt{\frac{d}{n}}$ holds with probability at least $1 - \exp{(-d/2)}$.
  Hence, we have
  \[
    \lambda_{\min}(AA^\top)
    \ge \left(1 - 2\sqrt{\frac{d}{n}}\right)^2 - 2(1+\varepsilon)\sqrt{\frac{d}{n}}
    = 1 - 2(3+\varepsilon)\sqrt{\frac{d}{n}} + \frac{4d}{n}
  \]
  with probability at least $1 - 2\exp{(-C_0\varepsilon^2 d)}  - \exp{(-d/2)} $.
  \qed
\end{proof}

When $n$ is large enough and $d = o(n)$, from Lemma~\ref{lemma:singular value of A}, $\lambda_{\min}(AA^\top) \gtrapprox 1$ holds with high probability.
Then, $\|\eta\|\le \delta$ holds with high probability.

Next, we confirm that the optimal solution to the original problem~\eqref{problem:max} can be obtained from \eqref{problem:approximated problem}.
Let $\eta^*$ denote the global solution. When $(\theta_k^\top \eta^*,\frac{1}{\sqrt{n}}P_k^\top \eta^*)$ is contained in $\mathcal{A}$, we can calculate the same maximal value as in the original problem~\eqref{problem:max} from \eqref{problem:approximated problem}.
We show that $(\theta_k^\top \eta^*,\frac{1}{\sqrt{n}}P_k^\top \eta^*)$ is contained in $\mathcal{A}$ with high probability. We have
\[
  \frac{(\theta_k^\top \eta^*)^2}{\|\theta_k\|^2} + \frac{1}{n} \|P_k^\top \eta^*\|^2
  \lelabeled{Lemma~\ref{lemma:properties of random matrix}.\ref{lemma:Johnson}}  \frac{(\theta_k^\top \eta^*)^2}{\|\theta_k\|^2} + \frac{d(1+\varepsilon)}{n}\|\eta^*\|^2
  \le \left(1 + \frac{d(1+\varepsilon)}{n}\right)\|\eta^*\|^2
\]
with probability at least $1 - 2\exp{(-C\varepsilon^2 d)}$.
Hence, with sufficiently large $n$ and $d = o(n)$, we have $\|A \eta^*\| \lessapprox  \|\eta^*\| \le \delta$ and $A\eta^*$ is contained in $\mathcal{A}$.
Note that the problem~\eqref{problem:max} is not convex optimization even if $\mathcal{H}$ is a convex set, because $-g$ is concave.

\begin{table}[tb]
  \caption{Details of the dataset for logistic regression~\cite{chang2011libsvm}.}
  \centering
  \begin{tabular}{c|ccc}
    Name & feature & class ($c$) & training size ($m$)\\
    \hline
    news20(binary) & 1,355,191 &2&19,996\\
    random & 1,000,000&2&100\\
    \hline
  \end{tabular}
\label{table:experiment2}
\end{table}

\begin{figure}[tb]
  \begin{minipage}{0.5\hsize}
    \centering
    \includegraphics[width=2.5 in]{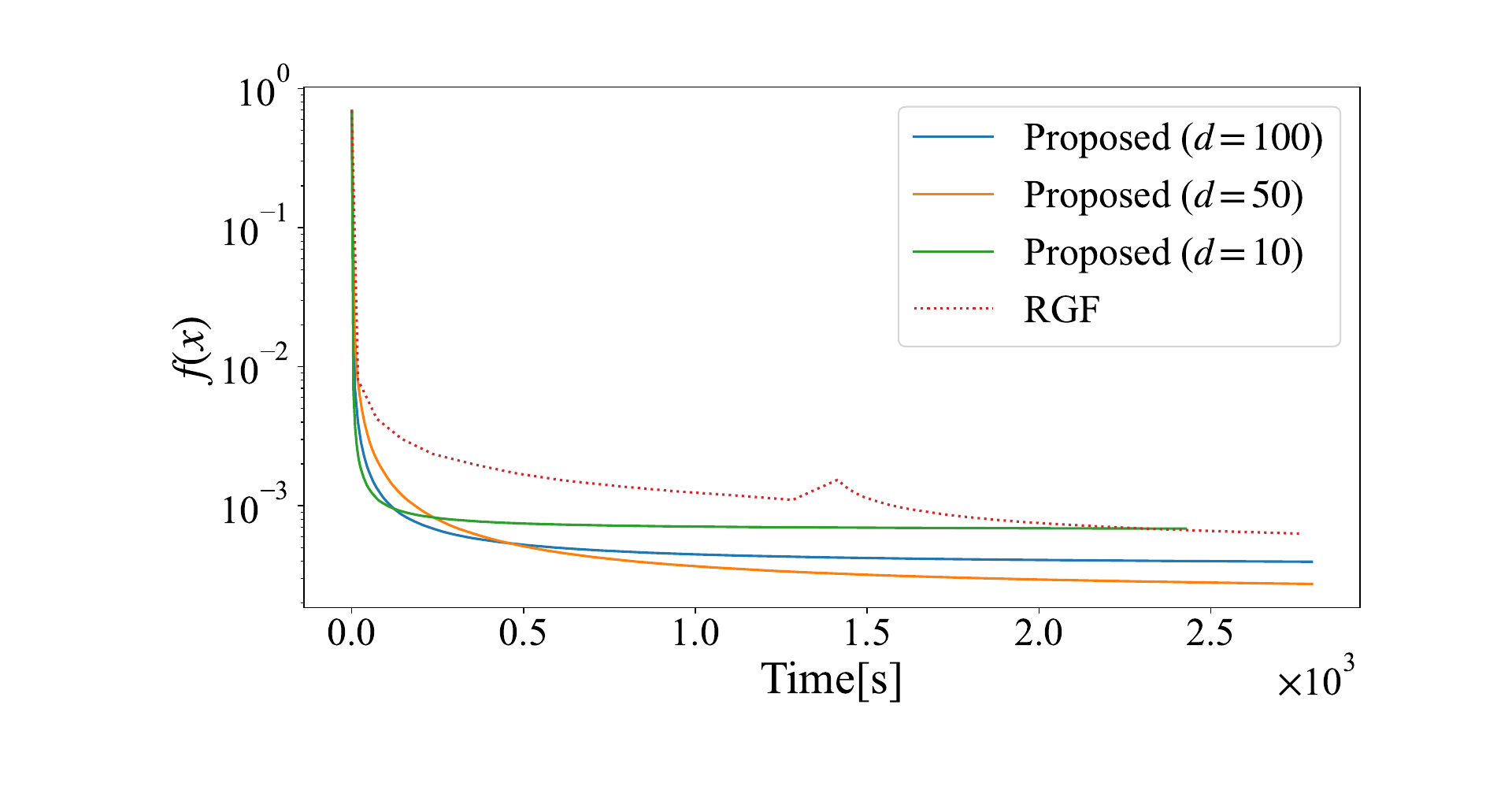}  
    \subcaption*{random Dataset ($\delta = 10^{-2}$)}  
  \end{minipage}
  \begin{minipage}{0.5\hsize}
    \centering
    \includegraphics[width=2.5 in]{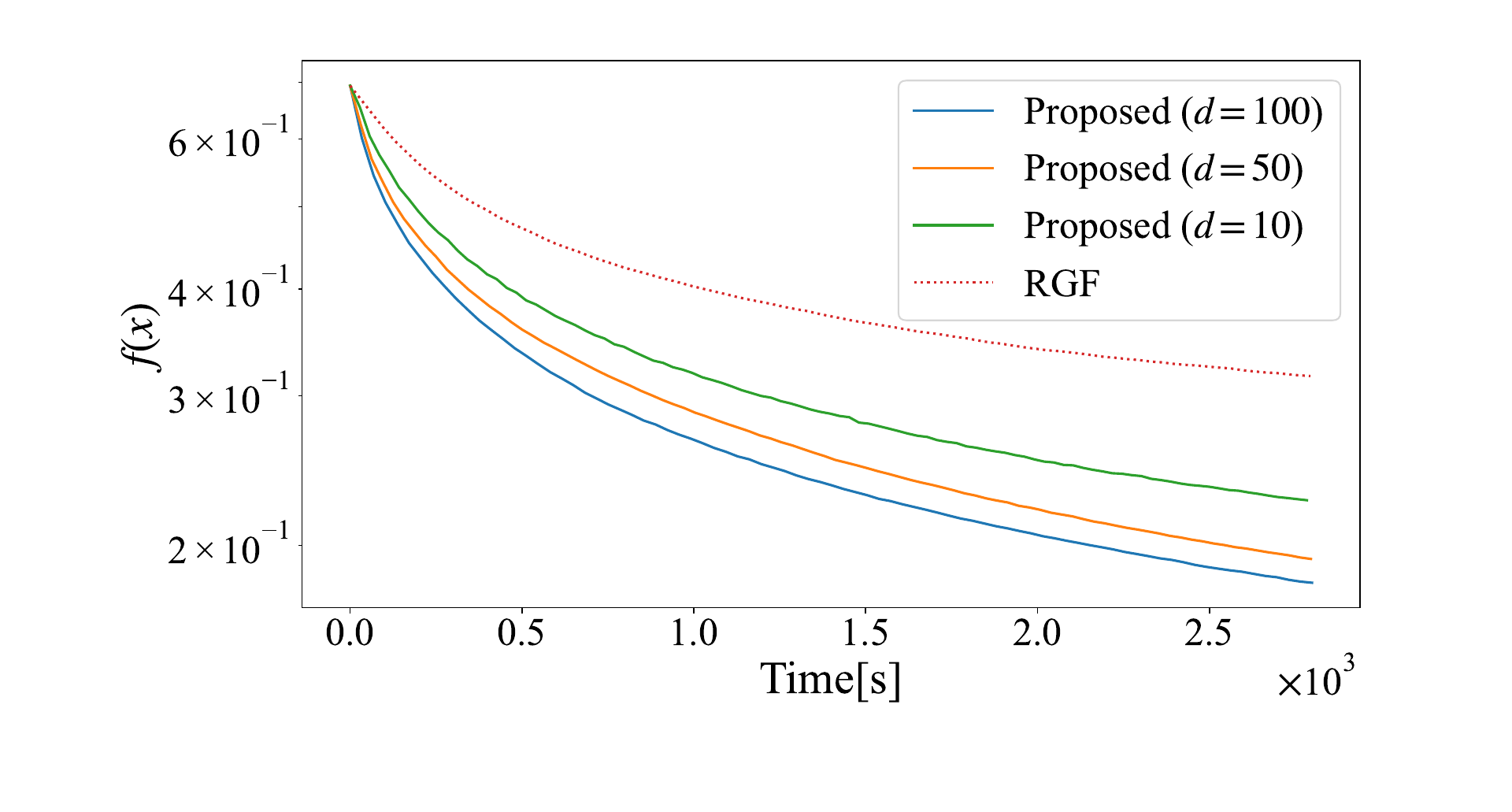}  
    \subcaption*{news20(binary) Dataset ($\delta = 10^{-3}$)}  
  \end{minipage}
  \caption{Adversarially robust logistic regression ($f(x)$ vs time).}
  \label{fig:experiment2}
\end{figure}

In numerical experiments, we evaluate $f(x)$ by solving the maximum optimization using the accelerated proximal gradient method until the norm of the generalized gradient is less than $10^{-7}$. 
In our proposed method, we evaluate $f(x + \frac{1}{\sqrt{n}} P^\top u)$ using the approximated random optimization problem~\eqref{problem:approximated problem}.
Furthermore, we increase the sampling size per iteration in both the proposed method and RGF, i.e., we update $x_k$ by
\begin{eqnarray*}
  x_{k+1} = x_k -\frac{\alpha_k}{\sqrt{l}}\sum_{i=1}^l\frac{{\color{black} h^{(k)}}(\mu_lu_k^{(l)}) - {\color{black} h^{(k)}}(-\mu_lu_k^{(l)})}{2\mu_k}P_k u_k^{(l)}, \\
  x_{k+1} = x_k -\frac{\alpha_k}{\sqrt{l}}\sum_{i=1}^l\frac{f(x_k + \mu_lv_k^{(l)}) - f(x_k-\mu_lv_k^{(l)})}{2\mu_k} v_k^{(l)},
\end{eqnarray*}
respectively.
We use a dataset from \cite{chang2011libsvm} and randomly generated one.
For the random dataset, we use a matrix $X$ and a vector $w$ whose entries are sampled from $\mathcal{N}(0,1)$.
The labels are generated as $y := \mathbf{1}_{x\ge 0}(Xw + \varepsilon)$, where $\varepsilon$ is a noise vector sampled from $\mathcal{N}(0,I_m/100)$.
In both methods, we set a smoothing parameter $\mu_k = 10^{-8}$ and use a fixed step size $\alpha_k = 10^{i}\;(i\in \mathbb{N})$. 

Figure~\ref{fig:experiment2} shows time versus the function values of RGF and the proposed methods for the datasets listed in Table~\ref{table:experiment2}.
From Figure~\ref{fig:experiment2} when evaluating $f(x)$ is time-consuming due to solve maximization problems, our proposed method converges faster than RGF.
This efficiency comes from the random projection technique, which leads to a reduction of function evaluation time.

\section{Conclusion}
\label{section:conclusion}
We proposed a new zeroth-order method combining random projections and smoothing method for non-smooth convex optimization problems. 
While our proposed method achieves $O(\frac{n}{\varepsilon^2})$ worst-case iteration complexity, which is equivalent to the standard result under convex and non-smooth setting,
 ours can converge with $O(\frac{d^2}{\varepsilon^2})$, which does not depend on the dimension $n$,
under some additional local properties of an objective. 
In numerical experiments, our method performed well when the function evaluation time can be reduced using random projection.
As discussed in Section~\ref{sec:softmax_exp},
since we have shown in this paper is the improvement of the ``worst-case'' oracle complexity, it is not always the case that the worst-case oracle complexity is achieved when the algorithm is actually run.
Indeed, many applications using zeroth-order methods have succeeded despite of large scale models and their oracle complexities depending on the dimension $n$~\cite{ye2018hessian,mania2018simple,pmlr-v80-choromanski18a,salimans2017evolution}. It may be interesting 
to investigate whether iteration complexities of zeroth-order methods are not affected by $n$ in practical use, or whether it can strongly depend on it in any problem setting  as a future work.
We also would like to investigate the convergence rate of our algorithm in a non-smooth and non-convex setting in the future.

\section*{Acknowledgement}
This work was supported by JSPS KAKENHI Grant Number 23H03351 and JST ERATO Grant Number JPMJER1903.

\section{Compliance with Ethical Standards}
This work was partially supported by JSPS KAKENHI (23H03351) and JST ERATO (JPMJER1903).
There is no conflict of interest in writing the paper.

\bibliographystyle{abbrvnat}
\bibliography{ref}
\end{document}